\documentclass[12pt]{article}
\title{Palm measures for Dirac operators and the $\Sineb$ process}
\date{}
\author{Benedek Valk\'o and B\'alint Vir\'ag}

\usepackage{amsmath, amsthm, amssymb}
\usepackage{graphicx}
\usepackage{amsmath, enumerate}
\usepackage{color, url}

\usepackage[longnamesfirst,round]{natbib}
\bibpunct[ ]{(}{)}{,}{a}{}{,}

\usepackage{hyperref}
    \newtheorem{theorem}{Theorem}
    \newtheorem{lemma}[theorem]{Lemma}
    \newtheorem{proposition}[theorem]{Proposition}
    \newtheorem{corollary}[theorem]{Corollary}

   \newtheorem{problem}[theorem]{Problem}
   
\theoremstyle{definition} 
    \newtheorem{definition}[theorem]{Definition}
    \newtheorem{remark}[theorem]{Remark}
    
    \newtheorem{assumption}{Assumption}
    
\newcommand{\eps}{\varepsilon}
\newcommand{\Z}{{\mathbb Z}}
\newcommand{\ZZ}{{\mathbb Z}}

\newcommand{\UU}{{\mathbb U}}

\newcommand{\R}{{\mathbb R}}
\newcommand{\CC}{{\mathbb C}}

\newcommand{\HH}{{\mathbb H} }

\newcommand{\lstar}{{\raise-0.15ex\hbox{$\scriptstyle \ast$}}}

\theoremstyle{remark} 
\newcommand{\Sineb}{\operatorname{Sine}_{\beta}}

\newcommand{\Sineop}{\mathtt{Sine}_{\beta}}
\newcommand{\Circop}{\mathtt{Circ}_{\beta,n}}

\newcommand{\Dirop}{\mathtt{Dir}}

\definecolor{violet}{rgb}{0.8,0,0.2}

\newcommand{\ed}{\stackrel{d}{=}}

\newcommand{\cA}{{\mathcal A}}

\newcommand{\cI}{{\mathcal I}}

\newcommand{\btau}{{\boldsymbol{\tau}}}

\newcommand{\TT}{\sigma} 

\newcommand{\mat}[4]{
\begin{pmatrix}
#1 & #2  \\
#3 & #4 
\end{pmatrix}
}

\newcommand{\bin}[2]{\left(
\begin{array} {c}
#1 \\
#2
\end{array}
\right)}

\newcommand{\ac}{{\text{\sc ac}}}
\newcommand{\dom}{\operatorname{dom}}

\newcommand{\ind}{\mathbf 1}

\newcommand{\benedek}[1]{\textcolor{red}{\tt{#1}}}

\newcommand{\res}{{\mathtt{r}\,}}

\newcommand{\uu}{\mathfrak  u}
\newcommand{\ttr}{\mathfrak t}

\newcommand{\bside}{\noindent\textbf{\benedek{Begin side computation.}}
\begin{footnotesize}}

\newcommand{\eside}{\end{footnotesize}
\noindent \textbf{\benedek{End side computation.}}}
\newcommand{\spec}{\operatorname{spec}}

\newcommand{\Poi}{\operatorname{Poi}}

\begin{document}
\maketitle
\begin{abstract}
We characterize the Palm measure of the $\Sineb$ process as the eigenvalues of an associated operator with a specific boundary condition. 
\end{abstract}


\section{Introduction}
The Palm measure of a point process is its distribution as seen from a typical point. We study the Palm measure of the $\Sineb$ process, which is the bulk scaling limit of the eigenvalues for a large class of random matrices. (See \cite{BVBV}, \cite{BVBV_sbo}, \cite{BEY}.)
There are several reasons to do this:  
\begin{itemize}
\item the secular function for (the operator associated to) the Palm measure has as simpler description than that of the original $\Sineb$ process,

\item the intensity of the Palm measure is given by the two-point correlation functions of the original process, a source of many open problems and conjectures,

\item the operator description of the Palm measure can be given explicitly: this is the  main goal of the current paper,

\item in the  case of the circular beta ensembles, the finite version of the $\Sineb$ process, the distribution of the Verblunsky coefficients can be given explicitly,

\item the Palm measures are closely related to the well-studied circular Jacobi ensembles and Hua-Pickrell measures of random matrix theory. We will describe the connection explicitly.
\end{itemize}

Let $b_1, b_2$ be independent two-sided standard Brownian motion, and for $t\in (0,1]$ set 
\begin{align}\label{eq:xy}
   y_t=e^{b_2(u)-u/2}, \quad  x_t=-\int_u^0 e^{b_2(s)-\tfrac{s}{2}} d b_1, \quad u=u(t)=\tfrac4{\beta} \log t.
\end{align}
Let $q$ be Cauchy distributed random variable independent of  $b_1, b_2$.
Define the function $R(t):(0,1]\to \R^{2\times 2}$ via
\begin{align}\label{eq:R_x_y}
R=\frac{X^t X}{2\det X}, \qquad X=\mat{1}{-x}{0}{y}.
\end{align}
In this paper, we consider the $\btau_\beta$ operator
\begin{equation}\label{tau_0}
\btau_\beta u=R^{-1} J u', \qquad J=\mat{0}{-1}{1}{0},
\end{equation}
acting on absolutely continuous of functions of the form $u:(0,1]\to \R^2$, with the boundary conditions $[1,0]^t$ at $0$ and $[-q,-1]^t$ at $1$. The spectrum of the $\btau_\beta$ operator is given by the $\Sineb$ process, see \cite{BVBV_sbo}. Our main result is the following theorem.

\begin{theorem}\label{thm:Sine_Palm}
The Palm measure of the $\Sineb$ process has the same law as the spectrum of $\btau_\beta$ with modified boundary condition $[1,0]^t$ at $1$. 
\end{theorem}

The proof of this theorem relies on two key facts: one is the shift invariance of the process and the other is the special relationship between the spectral measure and the eigenvalue distribution for the $\btau_\beta$ operator. Note that in this paper the spectral measure is always a real valued measure on $\R$.

For an $n\times n$ unitary matrix $U$ and a nonzero vector $v$ the spectral measure at $v$ is the probability measure 
$$
\sum_{\lambda} \langle v,\varphi_\lambda\rangle^2 \delta_\lambda. 
$$
Here the sum is over all eigenvalues $\lambda$, and the $\varphi_\lambda$ are the unit length eigenvectors of $U$ forming an orthonormal basis. 

For the Dirac operators of the type \eqref{tau_0}, we replace $\langle v,\varphi_\lambda\rangle$  by the value of $\varphi_\lambda$ at $0$ or $1$ to get the {\bf left} and {\bf right spectral measures}. More precisely, 
\begin{align}\label{eq:rightspec}
\mu_{\textup{left}}=\sum_{\lambda} |\varphi_\lambda(0)|^2\delta_\lambda,
\qquad
\mu_{\textup{right}}=\sum_{\lambda} |\varphi_\lambda(1)|^2\delta_\lambda,
\end{align}
where $\varphi_\lambda$ is the eigenfunction corresponding to $\lambda$ with normalization $\int_0^1 \varphi^t_\lambda R \varphi_\lambda dt=1$.  For the Dirac operators we consider the spectral measure has infinite total mass. 

An important ingredient in the proof of Theorem \ref{thm:Sine_Palm} is the description of how the right spectral measure of a  random Dirac operator changes when it is biased by the weight of zero. This is the content of the following proposition, see Proposition \ref{prop:Dir_Spec_weight} for a more precise formulation.  

\begin{proposition}\label{prop:Dir_Spec_weight_heur}
Let $\mu_q$ denote the right spectral measure of the  random Dirac operator of the form \eqref{tau_0} with boundary conditions $[1,0]^t$ and $[-q,-1]^t$, with $\mu_\infty$ corresponding to boundary conditions $[1,0]^t$ at both ends.
Assume that  $q$ is chosen independently of the function $R$ with a sufficiently regular distribution. Then the law of $\mu_q$ biased by $\mu_q((-\eps,\eps))$ has a distributional limit as $\eps\to 0$ given by $\mu_{\infty}$.


\end{proposition}

Another important ingredient of the proof of Theorem \ref{thm:Sine_Palm} is the following. 
\begin{proposition} \label{prop:sine_weights}
The weights of the right spectral measure of the  $\btau_{\beta}$ operator are independent of each other and from the sequence of eigenvalues. The weights  have gamma distribution  with shape parameter $\beta/2$ and mean 2. 
\end{proposition}

This is closely related to Theorem 9 in \cite{najnudel2021bead}, where the spectral measure is implicitly used to define the bead process, a Markov process with stationary distribution given by the $\Sineb$  point process. The connection is through the relationship between spectral measures and rank-one perturbations. Rather than exploring this connection, we deduce Proposition \ref{prop:sine_weights} from its finite-$n$ analogue. 


For $n\ge 1$ and $\beta>0$ the size $n$ circular beta-ensemble is the distribution of $n$ points on the unit circle with probability density 
\begin{align}
   \frac{1}{Z_{n,\beta}} \prod_{j<k\le n} |z_j-z_k|^\beta.
\end{align}
The \textbf{Killip-Nenciu measure} $\mu_{n,\beta}^{\textup{KN}}$ for given $n\ge 1$ and $\beta>0$ is the random probability measure where the support is given by the size $n$ circular beta-ensemble, and the weights  have Dirichlet$(\beta/2, \dots, \beta/2)$ joint distribution and are independent of the support.

In general, a finitely supported probability measure on the unit circle can be characterized by its modified Verblunski coefficients from the theory of orthogonal polynomials. When a probability measure is supported on $n$ points, all but the first $n$ modified Verblunski coefficients vanish. The first $n-1$ modified Verblunski coefficients are all in the unit disk $\UU=\{z:|z|<1\}$, with the last one being on the unit circle $\partial \UU$ (see \cite{KillipNenciu}, \cite{BNR2009}). 


The joint distribution of the modified Verblunski coefficients $\gamma_k, 0\le k\le n-1$ of the  Killip-Nenciu measure was described in \cite{KillipNenciu} and \cite{BNR2009}. It was shown that $\gamma_k, 0\le k\le n-1$ are independent,  $\gamma_k$  have density on $\mathbb U$ proportional to 
\begin{align}\label{eq:verb_circ}
(1-|z|^2)^{\frac{\beta}{2}(n-k-1)-1}, \qquad \text{for} \quad 0\le k \le n-2,
\end{align}
and $\gamma_{n-1}$ is uniform on $\partial \UU$.
Note that the distribution of $|\gamma_k|^2$ is Beta$(1,\frac{\beta}{2}(n-k-1))$, and $\gamma_k$ has rotationally invariant distribution.


In \cite{BVBV_sbo} it was shown that the spectral information of  a probability measure on the unit circle can also be described via a Dirac operator of the form \eqref{tau_0},  with $R$ defined from piece-wise constant paths  $x, y$ via \eqref{eq:R_x_y}, that are  built using the modified Verblunski coefficients.  See  Section \ref{sec:finite} and  Proposition \ref{prop:discrete_op_1} below. We extend the connection further  by the following lemma, see Proposition \ref{prop:Dir_discrete} for the precise formulation.

\begin{lemma}\label{lem:spectral_lift}
Consider a probability measure $\mu$ supported on $e^{i \lambda_j}, 1\le j\le n$, and the corresponding Dirac operator $\btau$. Then the left spectral measure of $\btau$ is the same as the lifting of $\mu$ to $\mathbb R$ stretched by $n$ and multiplied by $2n$: 
\[
\mu_{\textup{left}, \btau}(n \lambda_j+2n k \pi)=2n \mu(e^{i \lambda_j}), \qquad k\in \ZZ, 1\le j\le n.
\]
\end{lemma}

The factor $2$ is related to our convention in \eqref{eq:R_x_y}. This convention comes from the carousel representation, see \cite{BVBV_sbo}, in which the rotation speed is then given by the eigenvalue $\lambda$. Equivalently, this is related to the intensity $1/(2\pi)$ of the Sine$_\beta$ process. 


The Palm measure corresponding to a measure with  random Verblunski coefficients can be explicitly characterized in certain cases. The following is a special case of Proposition \ref{prop:path} below.

\begin{proposition} \label{prop:Verblunski_biasing} 
Let $\mu$ be a random probability measure supported on $n$ points on $\partial \UU$.
Assume that its modified Verblunsky coefficients $\gamma_i, 0\le i\le n-1$ are independent and have rotationally invariant law. Then the Palm measure of $\mu$ agrees with the $\eps\to 0$ limit of $\mu$ biased by the weight of the arc $\{e^{i\theta}, \theta\in (-\eps,\eps)\}$. The modified Verblunski coefficients of the Palm measure are given by 
\begin{align}\label{eq:gamma'}
\gamma'_i
=-\gamma_i\frac{1-\bar \gamma_i}{1-\gamma_i}, \qquad n=0,\ldots, n-2,\qquad \gamma'_{n-1}=1.
\end{align}
\end{proposition}
In the setting of probability measures on the unit circle the analogue of boundary condition for Dirac operators is given by the notion of Aleksandrov measures.  The modified Verblunski coefficients of Aleksandrov measures have a more complex description, leading to formulas such as \eqref{eq:gamma'}, see Section \ref{sec:Aleks}. 

Proposition \ref{prop:Verblunski_biasing} provides an explicit description of Verblunski coefficients of the Palm measure of the Killip-Nenciu measure as follows.


\begin{proposition}\label{prop:biased_V}
Let $\mu$ be chosen from the Killip-Nenciu law with parameters $\beta,n$, let $\tilde \mu$ be the uniform probability measure on the support of $\mu$, and
let $\nu$ be $\mu$ biased by the weight of 1. 
Let $X$ be picked from $\mu$, and  $Y$ be picked from the $\tilde \mu$.
 Then $\operatorname{supp} \mu(X\cdot)$,   $\operatorname{supp} \tilde \mu(Y\cdot)$, and $\operatorname{supp} \nu(\cdot)$ all have the same law. 
 The Verblunski coefficients $\gamma_k'$ of $\nu$ are independent and have density proportional to
\begin{align}\label{eq:verb_nu}
  (1-|z|^2)^{\frac{\beta}{2}(n-k-1)} |1-z|^{-2}, \qquad \text{for} \quad 0\le k \le n-2,  
\end{align}
and $\gamma_{n-1}'=1$.
\end{proposition}

The distribution of $\nu$ can be computed directly by noting that its support is the size $n$ circular beta-ensemble conditioned to have a point at 1, and the weights are just a size-biased version of the Dirichlet$(\beta/2, \dots, \beta/2)$ distribution.

The support of the measure $\nu$ has a point at 1, and if we remove that point then the joint density of the remaining points 
is proportional to
\begin{align}\label{eq:CJ}
\prod_{i<j\le n-1} |z_i-z_j|^\beta \prod_{j=1}^{n-1} |1-z_j|^\beta.
\end{align}
This 
is a circular Jacobi beta-ensemble of size $n-1$, with parameter $\delta=\beta/2$.

The weights of $\nu$ are independent of the support, and they are Dirichlet distributed with parameters $\beta/2$ ($n-1$ times) and $\beta/2+1$ (once). The last weight belongs to the point at $1$.


If we remove the point at 1 from the $\nu$ and renormalize the weights to sum up to 1, then the support and the weights are still independent, with the $n-1$ remaining weights distributed as Dirichlet$(\beta/2, \dots, \beta/2)$.
The modified Verblunski coefficients $\gamma_k, 0\le k\le n-2$ of this measure were described in \cite{BNR2009}. These are independent of each other, and the density of $\gamma_k$ on $\UU$ is proportional to 
$$
(1-|z|^2)^{\frac{\beta}{2}(n-k-2)-1}|1-z|^{\beta}, \quad \text{for} \quad 0\le k\le n-1,
$$
and the density of $\gamma_{n-2}$ on $\partial \UU$ is proportional to
$
|1-z|^{\beta}
$.
It would be interesting to find a direct derivation of this fact. 

\begin{problem}
Find a simple derivation of the joint distribution of the modified Verblunsky coefficients for the random probability measure obtained by removing the point 1 from $\nu$ (and renormalizing),  directly from Proposition  \ref{prop:biased_V}.
\end{problem}

\section{Dirac operators and their spectral measures}

\subsection{Dirac operator basics}

Let $R:\cI\to \mathbb R^{2\times 2}$ be a function  taking values in nonnegative definite matrices. Here $\cI$ is an interval with $\overline{\cI}=[0,\TT]$, that may be open at  one of its endpoints. Most often we will deal with the case $[0,1)$ or $(0,1]$. 

In this paper, a {\bf Dirac operator} is defined as   
\begin{equation}\label{tau}
{\boldsymbol{\tau}}u=R^{-1} J u', \qquad J=\mat{0}{-1}{1}{0},
\end{equation}
acting on some subset of functions of the form $u:\cI\to \R^2$.

\begin{assumption}

$R(t)$ is positive definite for all $t\in \cI$, $\|R\|, \|R^{-1}\|$ are locally bounded on $\cI$. Moreover, $\det R(t)=1/4$ for all $t\in \cI$.
\end{assumption}
If $R$ satisfies Assumption 1 then it can be parametrized as
\begin{align}
R=\frac{X^t X}{2\det X}  , \qquad X=\mat{1}{-x}{0}{y}, \qquad y>0, \,x\in \R.\label{eq:Rxy}
\end{align}
The assumption $\det R=1/4$ can be replaced by the more general condition $\int_{\cI} \det R(s)\, ds<\infty$. This setting is equivalent to ours up to a time change.

\begin{assumption}\ \\
\vspace{-2em}
\begin{itemize}
    \item If $\cI$ is open on the left then there is a nonzero $\mathfrak  u_0  \in \R^2$ so that  we have
\begin{align}
\int_{\cI} \left\|\mathfrak  u_0  ^t R \right\|ds<\infty,\qquad \int\limits_{\substack{s<t\\ (s,t)\in \cI}} \mathfrak  u_0  ^t R(s) \mathfrak  u_0   \, (J \mathfrak  u_0  )^t  R(t) J \mathfrak  u_0  ds dt<\infty.\label{eq:assmpns_1}
\end{align}
    
    \item If $\cI$ is open on the right then there is a nonzero  $\mathfrak  u_1  \in \R^2$ so that  we have
\begin{align}
\int_{\cI} \left\|\mathfrak  u_1  ^t R \right\|ds<\infty,\qquad \int\limits_{\substack{s>t\\ (s,t)\in \cI}} \mathfrak  u_1  ^t R(s) \mathfrak  u_1   \, (J \mathfrak  u_1 )^t  R(t) J \mathfrak  u_1   ds dt<\infty.\label{eq:assmpns_2}
\end{align}
\end{itemize}
\end{assumption}
Note that if $\cI$ is closed and $R$ satisfies Assumption 1 then \eqref{eq:assmpns_1} and \eqref{eq:assmpns_2} hold for any $\uu_0, \uu_1\in \R^2$.
Moreover, if $\bar I=[0,\TT]$, $0<t<\TT$ and $R$ satisfies Assumptions 1-2 with an appropriate $\uu_0$ or $\uu_1$ (if applicable), then it also satisfies the assumptions on $\cI\cap [0,t]$.

The Dirac operator $\btau$ with boundary conditions $\uu_0, \uu_1$ on $\cI$  is self-adjoint on an appropriate domain. Let $
L^2_R=L^2_R(\cI)$
denote the $L^2$ sub-space of $\cI\to \R^2$ functions with the squared norm 
\begin{align}
\|f\|_R^2=\int_{\cI} f^t(s) R(s) f(s) ds.
\end{align}
Let $\ac$ denote the subset of $\cI\to \R^2$ functions that are absolutely continuous. The following proposition follows from standard theory of Dirac operators (see e.g.~\cite{Weidmann}, \cite{BVBV_sbo}).

\begin{proposition}\label{prop:inverse_tau}
Let $\btau$ be a Dirac operator satisfying Assumption 1 on $\cI$. Let $\uu_0, \uu_1$ be non-zero vectors in $\R^2$ with which  the appropriate case of Assumption 2 is satisfied. Then $\btau$ is self-adjoint on the following domain:
\begin{align}
\dom_\btau=\{v\in L^2_R \cap  \ac\,:  \, \btau v\in L^2_R, \,\lim_{s\downarrow 0} v(s)^t\,J\,\uu_0   = 0, \,\,\lim_{s\uparrow \TT}v(s)^t\,J\,\uu_1 =0\}.
\end{align}
\end{proposition}
We denote the self-adjoint operator $\btau$ with boundary conditions
$\uu_0, \uu_1$ by $\Dirop(R_{\cdot},\uu_0  ,\uu_1 )$ or $\Dirop(x_{\cdot}+i y_{\cdot},\uu_0  ,\uu_1 )$. 


When $\uu_0\not \,\parallel \uu_1$ then we also make the following assumption, which is just a normalization. 

\begin{assumption}
$\uu_0^t J \uu_1=1$.
\end{assumption}
Note that if $\uu_0, \uu_1$ satisfy this assumption then we have $u_1=\frac{1}{\|u_0\|^2} J^t \uu_0+c \uu_0$ with $c\in \R$.

If $\uu_0\not\,\parallel \uu_1$ then $\btau=\Dirop(R_{\cdot},\uu_0  ,\uu_1 )$ has an inverse that is a Hilbert-Schmidt integral operator (see \cite{Weidmann} or \cite{BVBV_sbo}).
\begin{proposition}
Suppose that  $\btau=\Dirop(R_{\cdot},\uu_0  ,\uu_1 )$  with Assumptions 1-3 are satisfied. Then $\btau^{-1}$ is a Hilbert-Schmidt integral operator on $L^2_R$ given by
\begin{align}
    \btau^{-1} f(s)=\int_{\cI} K_{\btau^{-1}}(s,t) R(t) f(t) dt, \qquad K_{\btau^{-1}}(s,t)=\mathfrak  u_0   \mathfrak  u_1 ^t \ind(s<t)+\mathfrak  u_1  \mathfrak  u_0  ^t \ind(s\ge t),
\end{align}
and we have
\begin{align}
  \|\btau^{-1}\|_{\textup{HS}}^2=2\iint\limits_{\substack{s<t\\ (s,t)\in \cI}}
  \mathfrak  u_0  ^t R(s) \mathfrak  u_0   \, \uu_1^t R(t) \uu_1 ds dt.
\end{align}
\end{proposition}
Let $Y=\tfrac{X}{\sqrt{\det X}}$ with $X$ defined in \eqref{eq:Rxy}. Then the conjugated integral operator \begin{align}\label{def:res}
\res \btau:=Y \btau^{-1} Y^{-1}\end{align}
is self-adjoint on $L^2$ with integral kernel
\begin{align}\label{resint}
K_{\res \btau}(s,t)=\tfrac{1}{2}\,a_0(s) a_1(t)^t \ind(s<t)\;+\;\tfrac{1}{2}\,a_1(s) a_0(t)^t \ind(s\ge t),   
\end{align}
where 
\begin{equation}\label{eq:ac}
a_0=\frac{X\mathfrak  u_0  }{\sqrt{\det X}},\qquad a_1=\frac{X\mathfrak  u_1 }{\sqrt{\det X}}.
\end{equation}
We have  $\spec(\res \btau)=\spec(\btau^{-1})$, in particular $\|\res \tau\|^2_{\textup{HS}}=\|\btau^{-1}\|^2_{\textup{HS}}$.

\subsection{Secular function of a Dirac operator}

\cite{BVBV_szeta} introduced the  \emph{secular function} of a Dirac operator as a generalization of the (normalized) characteristic polynomial of a matrix. 
\begin{definition}
Suppose that the Dirac operator $\btau=\Dirop(R_{\cdot},\uu_0  ,\uu_1 )$ satisfies Assumptions 1-3 on $\cI$. We define its integral trace as
\begin{align}
    \ttr_\btau=\int_{\cI} \uu_0 R(s) \uu_1 ds,
\end{align}
and its secular function as
\begin{align}
    \zeta_{\btau}(z)=e^{-z\cdot {\mathfrak t}_\btau}{\det}_2(I-z \,\res \btau).
\end{align}
Here ${\det}_2$ is the (second) regularized determinant, see \cite{SimonTrace}, Chapter 9.
\end{definition}
The secular function of a Dirac operator $\btau$ is an entire function with zero set given by $\spec \btau$, and it is real on $\R$. As the next proposition shows, it can also be represented in terms of a canonical system. (See Proposition 13 in \cite{BVBV_szeta}.)

\begin{proposition}\label{prop:H_ODE}
Suppose that $\cI$ is closed on the right, and  $R$, $\uu_0$ satisfy Assumptions 1 and 2.
There is a  unique vector-valued function  $H: \cI\times \CC\to \CC^2$ so that for every $z\in \CC$  the function $H(\cdot, z)$ is the solution of the ordinary differential equation
\begin{align}
\label{eq:H}
J\frac{d}{dt}H(t,z)&=z R(t) H(t,z), \qquad t\in \cI, \qquad\lim_{t\to 0} H(t,z)=  \mathfrak  u_0  .
\end{align}
For any $t\in \cI$ the vector function $H(t,z)$ satisfies $\|H(t,z)\|>0$, and its two entries are entire functions of $z$ mapping the reals to the reals.

In addition, if $\btau=\Dirop(R_{\cdot},\uu_0  ,\uu_1 )$ with
$\uu_1$ satisfying Assumption 3,  then
\begin{align}\label{eq:zeta}
    \zeta_{\btau}(z)=H(1,z)^t J \uu_1.
\end{align}
\end{proposition}
The proposition immediately implies that under the listed conditions for any $0<t<\TT$ we have
\[
 \zeta_{\btau,t}(z)=H(t,z)^t J \uu_1,
\]
where $\zeta_{\btau,t}$ is the secular function of $\btau$ defined on $\cI\cap [0,t]$ with boundary conditions $\uu_0, \uu_1$.

As the next proposition shows, the function $H(t, \cdot)$ is continuous on compacts in $t$ at $t=0$, and it depends  continuously on $R$ as well. We use the notation $\ttr_{\btau,t}$ and $\res \btau_t$ for the integral trace and resolvent of $\btau$ restricted to $\cI \cap [0,t]$.

\begin{proposition}\label{prop:H_cont}
Suppose that $\cI$ is closed on the right, and  $R$, $\uu_0$ satisfy Assumptions 1 and 2, and let $H$ be the solution of \eqref{eq:H}. Suppose that  $\uu_1$ satisfies Assumption 3, and let $\btau$ be the Dirac operator defined by $R, \uu_0, \uu_1$. Then there is an absolute constant $c>1$ (depending only on $\uu_0$) so that for all $t\in \cI$, $z\in \CC$ we have
\begin{align} \label{eq:H_cont}
    |H(t,z)-\uu_0|\le \left(c^{|z|(|\ttr_{\btau,t}|+\|\res \btau_t\|+\int_0^t |a_0(s)|^2 ds)}-1\right) c^{\left(|z|(|\ttr_{\btau,t}|+\|\res \btau_t\|+\int_0^t |a_0(s)|^2 ds)+1\right)^2}.
\end{align}
Suppose that $\tilde R$, $\uu_0, \tilde \uu_1$ also satisfy  Assumptions 1-3 on the same $\cI$, let $\widetilde H$ be the solution of \eqref{eq:H}, and $\tilde \btau$ the corresponding Dirac operator. 
Then there is an absolute constant $\uu_0$ (depending only on $\uu_0$) so that for all $t\in \cI$, $z\in \CC$ we have
\begin{align}\notag
&    |H(t,z)-\tilde H(t,z)|\le \left(c^{|z|\left(|\ttr_{\btau,t}-\tilde \ttr_{\btau,t}|+\|\res \btau_t-\res \tilde\btau_t\|+\sqrt{\int_{0}^t|a_0(s)-\tilde a_0(s)|^2 ds \int_{0}^t(|a_0(s)|^2+|\tilde a_0(s)|^2)ds
}\right)}-1\right)\\&\hskip150pt \times
c^{\left(|z|(|\ttr_{\btau,t}|+|\ttr_{\tilde\btau,t}|+\|\res \btau_t\|+\|\res \tilde\btau_t\|+\int_0^t \left(|a_0(s)|^2+|\tilde a_0(s)|^2 \right)ds)+1\right)^2}.\label{eq:H1_cont}
\end{align}
\end{proposition}
\begin{proof}
Theorem 9.2(c) of \cite{SimonTrace} shows that  if $\kappa_1, \kappa_2$ are Hilbert-Schmidt  operators on the same domain then
\begin{align}\label{det2_tri}
|{\det}_2(I-z \kappa_1)-{\det}_2(I-z \kappa_2)|\le |z|\cdot  \|\kappa_1-\kappa_2\|_2 c^{ |z|^2 (\|\kappa_1\|_2^2+ \|\kappa_2\|_2^2)+1}
\end{align}
with an absolute constant $c$. Using \eqref{eq:zeta} one readily obtains a bound of the forms \eqref{eq:H_cont}, \eqref{eq:H1_cont} for $|H(t,z)^tJ\uu_1-1|$ and $|H(t,z)^tJ\uu_1-\tilde H(t,z)^tJ\tilde \uu_1|$, without the $a_0, \tilde a_0$ integral terms. 

Note that Assumptions 1-3 also hold for $R, \uu_0, \uu_1+\uu_0$ (and similarly for $\tilde R, \uu_0, \uu_1+\uu_0$). These yield bounds for $|H(t,z)^tJ(\uu_0+\uu_1)-1|$ and $|H(t,z)^tJ(\uu_0+\uu_1)-\tilde H(t,z)^tJ(\uu_0+\tilde \uu_1)|$, but with modified integral trace and $\res \btau$. By changing the right boundary condition by $\uu_0$ we change the integral trace on $\cI\cap [0,t]$ by $\int_0^t |a_0(s)|^2 ds$, and the integral kernel of the conjugated integral operator \eqref{resint} by the function $K_0(s,t)=a_0(s) a_0(t)^t$. We have 
\[
\|K_0\|_{\textup{HS}}^2=\iint_{\cI\times \cI}\|a_0(s) a_0(t)^t\|^2ds dt=\left(\int_\cI |a_0(s)|^2 ds\right)^2
\]
and
\begin{align*}
\|K_0-\tilde K_0\|^2_{\textup{HS}}&=\iint_{\cI\times \cI} \|a_0(s) a_0(t)^t-\tilde a_0(s) \tilde a_0(t)^t\|^2ds\\&\le 2\int_{\cI}|a_0(s)-\tilde a_0(s)|^2\int_{\cI}(|a_0(s)|^2+|\tilde a_0(s)|^2) ds.
\end{align*}
This implies the bounds  \eqref{eq:H_cont}, \eqref{eq:H_cont} for the terms 
\begin{align*}
    |H(t,z)^tJ(\uu_0+\uu_1)-1|, \qquad |H(t,z)^tJ(\uu_0+\uu_1)-\tilde H(t,z)^tJ(\uu_0+\tilde \uu_1)|,
\end{align*}
(now with the $a_0, \tilde a_0$ terms included), from which the proposition follows.
\end{proof}

The ODE system \eqref{eq:H} also provides a representation of the eigenfunctions of $\btau$.

\begin{proposition}\label{prop:H_L2}
Suppose that $\cI$ is closed from the right, and  $R$, $\uu_0$ satisfy Assumptions 1 and 2.
The function $H(\cdot,\lambda)$ given in \eqref{eq:H} is in $L^2_R$ for any fixed $\lambda\in \R$. In particular, if $\btau=\Dirop(R_{\cdot}, \uu_0, \uu_1)$ for any nonzero $\uu_1\in \R^2$ then the eigenvalues of $\btau$ are given by the zero set of  
$H(1,z)^tJ\uu_1$. Moreover, if $\lambda$ is an eigenvalue then  $H(\cdot,\lambda)$ is the corresponding eigenfunction.
\end{proposition}
\begin{proof}
By differentiating \eqref{eq:H} in $z$ we get
\[
J \partial_t \partial_z H=RH+z R \partial_z H, \quad \text{and} \quad
\partial_t (H^t J \partial_z H)=H^t R H.
\]
For $0<\eps<\TT$ and $\lambda\in \R$ we get
\[
H^tJ\partial_\lambda H(\TT)-H^tJ\partial_\lambda H(\eps)=\int_\eps^{\TT} H^t R H ds>0.
\]
We let $\eps\to 0$. By Propositions \ref{prop:H_ODE} and  \ref{prop:H_cont} the function $H$ is  continuous in $t$ at $0$ in the uniform-on-compacts topology, and is analytic for each $t$. It follows that $\partial_\lambda H$ is continuous at $t=0$, hence the left side converges to $H^tJ\partial_\lambda H(\TT)$
Using monotone convergence for the right hand side, we get
\begin{align}\label{eq:H_L2}
 0<\int_{\cI} H^t R H ds=H^tJ\partial_zH(\TT)\in \R.   
\end{align}
This proves that $H(\cdot,\lambda)\in L^2_R$. It also implies that if for $\lambda\in \R$ we have $\uu_1 \parallel H(\TT,\lambda)$ then $\lambda$ is an eigenvalue of $\tau$ with eigenfunction $H(\cdot, \lambda)$.

If $\lambda\in \R$ is an eigenvalue with eigenfunction $f$ then $\btau f=\lambda f$ and $f$ satisfies  the boundary conditions at $t=0$ and $t=\TT$ given in the definition of $\dom_\btau$. By classical theory of differential operators (see \cite{Weidmann}) the equation $\btau g=\lambda g$ has at most one $L^2_R$ solution with a given boundary condition at $t=0$. Since $H( \cdot,\lambda)$ and $f$ are both solutions and have the same boundary conditions, they must be parallel, which means that $H( \cdot, \lambda)$ is the eigenfunction and $\uu_1 \parallel H(\TT,\lambda)$. In particular, the zero set of $H(1,z)^tJ\uu_1$ is the same as the spectrum of $\btau$.
\end{proof}

\begin{corollary}\label{cor:spec} When $\btau=\Dirop(R,\uu_0,\uu_1)$ satisfies Assumptions 1,2 then the eigenfunctions are continuous at both endpoints of $\mathcal I$. 
\end{corollary}
\begin{proof}
If  $\mathcal I$ is closed from the right then these statements follow immediately from Proposition \ref{prop:H_L2}. (Note that by \eqref{eq:H_L2} the function $H(1, \cdot)$ cannot be constant.)
If $\mathcal I$ is closed from the left then we can just reverse time. 
\end{proof}

\subsection{Spectral measure}

In this section we will work with the case $\bar \cI=[0,1]$.
Consider a Dirac operator $\btau=\Dirop(R_{\cdot}, \uu_0,\uu_1)$ 
satisfying Assumptions 1 and 2.
 Denote the spectrum of $\btau$ by $\spec \btau$, and let $f_\lambda$ denote an  eigenfunction corresponding to an eigenvalue $\lambda$. 

\begin{definition}\label{def:spectral}
The {\bf left} and {\bf right spectral measures} of $\btau$ are defined as
\begin{align}
   \mu_{\textup{left}}=\sum_{\lambda\in \spec \tau} \frac{f_\lambda(0)^t  f_\lambda(0)}{\|f_\lambda\|_R^2} \delta_\lambda, \qquad
   \mu_{\textup{right}}=\sum_{\lambda\in \spec \tau} \frac{f_\lambda(1)^t  f_\lambda(1)}{\|f_\lambda\|_R^2} \delta_\lambda.
\end{align}
\end{definition}
By Corollary \ref{cor:spec} these measures are well-defined.  If $\cI$ is closed from the right then $f_\lambda(\cdot)$ can be chosen as $H(\lambda, \cdot)$ defined in Proposition \ref{prop:H_ODE}, see also Proposition \ref{prop:H_L2}. 



\begin{definition}\label{def:phase}
Let $\cI$ be closed on the right, and $\uu_0=[1,0]^t$. Suppose that 
that $R$, $\uu_0$ satisfy Assumptions 1 and 2, and consider $H$ from \eqref{eq:H}. Denote $H(1,z)=[A(z),B(z)]^t$, and define the phase function of $R$ as 
\begin{align}\label{eq:alpha}
\alpha(t, \lambda)=2 \Im \log(A(\lambda)-i B(\lambda)).
\end{align}
Here we take the continuous branch of logarithm so that $\alpha(t,0)=0$ for $t\in \cI$.
\end{definition}

The next result shows how we can identify the spectral measure of a Dirac operator using the phase function.

\begin{lemma}\label{lem:phase_spectral1}
Let  $\cI$ be closed on the right, $\uu_0=[1,0]^t$, and  $\uu_1=[-q,-1]$ or  $\uu_1=[1,0]$. Suppose that $R, \uu_0$  satisfy   Assumptions 1 and 2, let $\btau=\Dirop(R_{\cdot},\uu_0,\uu_1)$, and consider the phase function $\alpha$ given in Definition \ref{def:phase}. Then we have 
$$
\mu_{\textup{right}}=2\sum_{k\in \mathbb Z}(\alpha^{-1})'(2\pi k +u) \delta_{\alpha^{-1}(2\pi k+u)}
$$
where $\cot(u/2)=-q$ if $\uu_1=[-q,-1]$ and $u=0$
 if $\uu_1=[1,0]$. The weight of an eigenvalue $\lambda$ is given by 
\[ \mu_{\textup{right}}(\lambda)=\frac{2}{\partial_\lambda \alpha(1,\lambda)}
= \frac{A^2(\lambda)+B^2(\lambda)}{A'(\lambda) B(\lambda)-A(\lambda)B'(\lambda)}.
\]
\end{lemma}

\begin{proof} 
By Proposition \ref{prop:H_L2} 
the spectrum of $\btau$ is given by the zero set of the function $H(1,z)^tJ\uu_1$. We have  $H(1,z)^tJ\uu_1=A(z)-qB(z)$ if $\uu_1=[-q,-1]$, and  $H(1,z)^tJ\uu_1=B(z)$ if $\uu_1=[1,0]$. By Definition \ref{def:phase} the zero set of $H(1,z)J\uu_1$ is exactly the set $\{\alpha^{-1}(2\pi k+u): k\in \ZZ\}$. 

If $\lambda\in \R$ is an eigenvalue then by \eqref{eq:H_L2} of Proposition \ref{prop:H_L2} we have
\[
\int_0^1 H^t R Hds= H^t J \partial_z H(1)=A'(z)B(z)-A(z)B'(z). 
\]
This gives the expression
\[ \mu_{\textup{right}}(\lambda)=\frac{H(\lambda,1)^t H(\lambda,1)}{H^t(1,\lambda)J\partial_{\lambda}H(1,\lambda)}= \frac{A^2(\lambda)+B^2(\lambda)}{A'(\lambda) B(\lambda)-A(\lambda)B'(\lambda)}.
\]
We also have
\[
\partial_\lambda \alpha(1,\lambda)=2 \Im \frac{A'-i B'}{A-i B}= 2\frac{A'(\lambda) B(\lambda)-A(\lambda)B'(\lambda)}{A^2(\lambda)+B^2(\lambda)},
\]
which proves $
\mu_{\textup{right}}(\lambda)=\frac{2}{\partial_\lambda \alpha(1,\lambda)}$.
\end{proof}

\begin{lemma}\label{lem:spec_convergence} Assume that $\cI$ is closed on the right. Let $\btau$, $\{\btau_n, n\ge 1\}$ be Dirac operators on $\cI$ with common left boundary condition $\uu_0$ 
satisfying Assumptions 1 and 2. Assume that 
$$\|\res \tau_n -\res \tau\|_{\textup{HS}}\to 0, \qquad  \ttr_{\tau_n}\to \ttr_{\tau},  \qquad \int_{\cI} |a_0(s)-a_{0,n}(s)|^2ds\to 0,$$ 
as $n\to \infty$.
Then in the vague topology of measures, 
\begin{align}
    \mu_{\textup{right},n}\rightarrow \mu_{\textup{right}}, \qquad  \mu_{\textup{left},n}\rightarrow \mu_{\textup{left}}.
\end{align}
\end{lemma}

\begin{proof}

From $\res \tau_n\to \res \tau$ we get $\lambda_{k,n}\to \lambda_k$. From the continuity bound of Proposition \ref{prop:H_cont} we get that $H_n\to H$ uniformly on compacts. This implies that we must have $\uu_{1,n}\to \uu_1$. Indeed, assume  that there is a subsequence with a different limit $\mathfrak{v}$.  Then the zero set of $\uu_{1,n}J H_n$ (the spectrum of $\btau_n$) would converge locally on compacts to the
 zero set of $\mathfrak{v}^t J H$ by Hurwitz's theorem, which would contradict $\lambda_{k,n}\to \lambda_k$.

To prove that the right spectral weights converge, we use the formula   given by Lemma \ref{lem:phase_spectral1}.
Since $A-iB$ has no real zeros, we can extend the definition of $\alpha$ from \eqref{eq:alpha} to a neighborhood of $\R$. We also get that $\alpha_n\to \alpha$ uniformly on compacts in some neighborhood of $\mathbb R$. Since $\alpha$ is analytic, this implies convergence of the derivatives. Hence from Lemma \ref{lem:phase_spectral1} we get $   \mu_{\textup{right},n}(\lambda_{k,n})\rightarrow \mu_{\textup{right}}(\lambda_k)$, proving the lemma for the right spectral measure. Since $H_n(\TT,\cdot)\to H(\TT,\cdot)$ uniformly on compacts, this implies that $\|H_n(\cdot, \lambda_{k,n})\|_{R_n}^2\to \|H_n(\cdot, \lambda_{k})\|_{R}^2$, which in turn shows $\mu_{\textup{left},n}\rightarrow \mu_{\textup{left}}$, as $H_n(0,\lambda)=H(0,\lambda)=\uu_0$.
\end{proof}

\subsubsection*{Basic transformations, path reversal and spectral measure}

We introduce the projection operator
$
\mathcal P\binom{z_1}{z_2}=\frac{z_1}{z_2},
$
with the natural extension $\mathcal P \binom{a}{0}=\infty\in \partial \HH$ for a nonzero real $a$. 
For the boundary conditions of an operator $\Dirop(R_{\cdot},\uu_0,\uu_1)$ only the directions of $\uu_0$, $\uu_1$ matter, hence we can also use $\mathcal P \uu_0, \mathcal P \uu_1$ to identify the boundary conditions.

A $2\times 2$  matrix $A$ with a nonzero determinant can be identified with a linear fractional  transformation  via $z\to \mathcal P A \binom{z}{1}$. For an $A$ with real entries the corresponding linear fractional transformation is an isometry of the hyperbolic plane $\HH$. We record the form of the hyperbolic rotation of $\HH$ about $i$ taking $r\in \R$ to $\infty$ both as a $2\times 2$ matrix and the corresponding linear fractional transformation.
\begin{align}\label{def:hyprot}
  T_r=\frac{1}{\sqrt{1+r^2}}\mat{r}{1}{-1}{r}, \qquad   \mathcal{T}_r(z)=\mathcal P T_r \bin{z}{1}=\frac{rz+1}{r-z}.
\end{align}
The following lemma summarizes how an isometry of $\HH$ (in particular, a hyperbolic rotation) acts on a Dirac operator, the statements follow from the definitions.

\begin{lemma}\label{lem:isometry}
Suppose that $Q\in \R^{2\times 2}$ with $\det Q=1$ and define $\mathcal Q: \HH\to \HH$ via $\mathcal Q z=\mathcal P Q\binom{z}{1}$. Let  $\btau=\Dirop(R_{\cdot}, \uu_0, \uu_1)=\Dirop(x+i y, \uu_0, \uu_1)$ be a Dirac operator satisfying Assumptions 1 and 2. Then $\tilde \btau=Q \btau Q^{-1}=\Dirop(\tilde R, \tilde \uu_0, \tilde \uu_1)=\Dirop(\tilde x+i \tilde y, \tilde \uu_0, \tilde \uu_1)$ with 
\[
\tilde R=JQJ^{-1}R Q^{-1}, \quad \tilde x+ i \tilde y=\mathcal{Q}(x+i y), \quad \tilde u_j=Q u_j.
\]
The operator $\tilde \btau$ satisfies Assumptions 1-2, and if $\btau$ satisfies Assumption 3 then the same holds for $\tilde \btau$. 
If $Q=T_r$ for some $r\in \R$ then the right and left spectral measures of $\btau$ agree with those of $\tilde \btau$.
\end{lemma}

  Let $\rho f(t)=f(1-t)$ be the time-reversal operator on functions defined on $[0,1)$ or $(0,1]$. Let $S=\mat{1}{0}{0}{-1}$, and $\mathfrak{r}: \CC\to \CC$ be defined as $\mathfrak{r}(z)=-\bar z$. The following lemma is a simple consequence of the definitions. 
  
  \begin{lemma}\label{lem:reversal}
  Suppose that $\btau=\Dirop(z,\uu_0,\uu_1)$ is a Dirac operator satisfying Assumptions 1 and 2. Then $
  \tilde \btau=\rho^{-1} S \btau S \rho
 $
also satisfies Assumptions 1 and 2. We have  $\tilde \btau=\Dirop(\tilde z, \tilde \uu_0, \tilde \uu_1)$ with $\tilde z=\rho \mathfrak{r} z$, $\tilde \uu_0=-\tilde uu_1$, and $\tilde \uu_1=-\tilde uu_0$. Moreover, 
\[
\mu_{\textup{left},\btau}=\mu_{\textup{right},\tilde\btau}, \qquad \mu_{\textup{right},\btau}=\mu_{\textup{left},\tilde\btau}.
\]
  \end{lemma}

\section{Biasing by the weight at zero and the Palm measure}\label{sec:biasing}

In this section we discuss the notion of conditioning a random measure to charge location $x$: ``biasing by the weight of $x$''. For  translation-invariant random measures, this coincides with the notion of the Palm measure. 
\begin{definition}\label{def:biased} Given a random measure $\mu$ on a metric space $S$, let $\mu_\eps$ have Radon-Nikodym derivative $$\frac{\mu(B_\eps(x))}{E\mu(B_\eps(x))}$$
with respect to $\mu$. That is, $\mu_\eps$ is $\mu$ biased by the weight of $B_\eps(x)$.
We define {\bf $\mu$ biased by the weight of  $x\in S$} as the $\eps\to 0$ distributional limit of  $\mu_\eps$ if it exists.
\end{definition}

\begin{remark}
If $\mu$ is a random probability measure, and $X$ is a random pick from $\mu$, then $\mu$ biased by the weight of $x$ can be interpreted as the posterior distribution of $\mu$ given $X=x$ in the Bayes sense. 
\end{remark}

Often a slightly more general version of Definition \ref{def:biased} is needed, where the metric space is {\bf decorated} with a random function $\xi:S\to R$, and arbitrary measure space. The {\bf decoration} $\xi$ seems unnecessary at first, but is a standard tool in Palm measure theory. It will be  useful in defining Palm measures for operators. 

\begin{definition}\label{def:biased1} Let $\mu$ be a random measure on a metric space $S$, and let $\xi$ be a random measurable function $S\to R$, where $R$ is a measure space. 
Let $(\mu_\eps,\xi_\eps)$ have Radon-Nikodym derivative $$\frac{\mu(B_\eps(x))}{E\mu(B_\eps(x))}$$
with respect to $(\mu,\xi)$. That is, $(\mu_\eps,\xi_\eps)$ is $(\mu,\xi)$ biased by the weight of $B_\eps(x)$.
We define {\bf $(\mu,\xi)$ biased by the weight of  $x\in S$} as the $\eps\to 0$ distributional limit of  $(\mu_\eps,\xi_\eps)$ if it exists.
\end{definition}

When $\xi$ is constant, it can be ignored and Definition \ref{def:biased1} is equivalent to   Definition \ref{def:biased}. \medskip

Next, we define the Palm measure of a decorated, translation invariant random measure with finite intensity.  
One could  ignore the decoration $\xi$ for the first reading, and think of it as almost surely constant. Even in the general case, it is good to think of $\xi$ as ``coming along for the ride''.

Let $\mu$ be a random measure on $\mathbb R$, and let $\xi$ be a random function from $\mathbb R$ to some measure space $R$. We assume that $(\mu,\xi)$ is translation invariant:
\[
(\mu(\cdot +t),\xi(\cdot+t))\ed(\mu,\xi), \qquad \text{for all }t\in \mathbb R, 
\]
and that for some positive $\rho$  \begin{equation}\label{e:mua}E\mu(A)=|A|\rho \qquad \text{for any Borel set $A\subset \mathbb R$.}
\end{equation}
The number $\rho$ is called the {\bf intensity} of $\mu$.

\begin{definition}\label{def:palm}
Let $A$ be a subset of $\mathbb R$ with positive and finite Lebesgue measure. Let $(\mu_A,\xi_A)$ be distributed as the biasing of $(\mu,\xi)$ by $\mu(A)$, namely for any measurable set $M$  we have 
$$
P((\mu_A,\xi_A)\in M)=\frac{E[\mu(A);\,(\mu,\xi)  \in M]}{E\mu(A)}
$$
Given $\mu_A$, let $X$ be a random point picked from the probability measure $\mu_A(A\cap \cdot)/\mu_A(A)$, the normalized version of $\mu_A$ restricted to $A$. Then the {\bf Palm measure} $(\mu_*,\xi_*)$ of $(\mu,\xi)$ is defined as the translation $(\mu_A,\xi_A)(\cdot-X)$. 
\end{definition}
When $\xi$ is constant, in words, the Palm measure is the re-centering of the $\mu(A)$-biased version $\mu_A$ of $\mu$ at a random point picked from the restriction of $\mu_A$ to $A$.

Definition \ref{def:palm} can be summarized in terms of bounded test functions $f$  as follows:
$$
Ef(\mu_*,\xi_*)=\frac{1}{E\mu(A)}E\int_A f(\mu(\cdot -x),\xi(\cdot-x)) \mu(dx).
$$
which agrees with the classical definition, see  \cite{Kallenberg}, Chapter 11, equation (1). There, in Lemma 11.2 it is shown that the definition does not depend on the set $A$. 

\begin{proposition} \label{prop:Palm} Let $(\mu,\xi)$ be translation-invariant decorated random measure on $\mathbb R$, and assume that $E\mu\{[0,1]\}=\rho<\infty$.  Then $(\mu,\xi)$ biased by the weight of $0$ exists and equals the Palm measure of $(\mu,\xi)$. 
\end{proposition}

\begin{proof}[Proof of Proposition \ref{prop:Palm}]
By Definition \ref{def:biased1}, for each $\eps>0$ the Palm measure is given by 
$\nu=\mu_\eps(\cdot-X_\eps)$ where $X_\eps$ is picked from a probability measure supported on $B_\eps(0)$. Equivalently, 
$$
\mu_\eps=\nu(\cdot+X_\eps).
$$
Since $\nu$ does not depend on $\eps$, and $X_\eps$ converges in law to $0$, the pair $(\nu,X_\eps)$ converges in law to $
(\nu,0)$. By the continuity theorem, $\mu_\eps=\nu(\cdot+X_\eps)$ converges to $\nu(\cdot+0)=\nu$.
\end{proof}

\begin{remark}\label{rem:circle}
In Definition \ref{def:palm} we defined the Palm measure for decorated translation invariant random measures on $\R$ with a finite intensity. The same setup (with the appropriate modifications) works for rotation invariant random measures on the unit circle with a finite expected total weight. In particular, Proposition \ref{prop:Palm} holds for such random measures.
\end{remark}

Let $\mathbb S=[0,2\pi)$ with the topology of the circle, and let $\{\cdot \}_{2\pi}$ be the standard covering map $\mathbb R\to \mathbb S$. In other words, $\{x \}_{2\pi}=x-2\pi \lfloor x/(2\pi)\rfloor$, the single element of   $[0,2\pi)\cap (x+2\pi \ZZ)$.

\begin{lemma}\label{lem:pullback}
Let $\kappa>0$ and $\lambda$ be a differentiable random bijection $\R\to \R$, and set $V=\{\lambda^{-1}(0)\}_{2\pi}$. Let $\gamma$ be a random variable taking values in a Polish space $\mathcal S$.
For $u\in \mathbb R$, consider the measure 
\begin{align}\label{def:Xi}
\Xi_u=\kappa \sum_{w\in 2\pi \ZZ+u} \lambda'(w)\delta_{\lambda(w)}.
\end{align}
Let $U$ be independent of $(\lambda,\gamma)$, and assume that the law of $\{U\}_{2\pi}$ is absolutely continuous with bounded density $f$ on $\mathbb S$. Moreover, assume that on the support of the law of the random variable $V$ the density $f$ is positive and  continuous on $\mathbb S$. 

Then $(\Xi_U,(U,\gamma))$ biased by the weight of zero exists and has the law of $(\Xi_{V},(V,\gamma))$ biased by $f(V)$. In particular, if $U$ is uniform on $\mathbb S$ or $\lambda(0)=0$ a.s., then $(\Xi_U,(U,\gamma))$ biased by the weight of zero has the law of $(\Xi_{V},(V,\gamma))$. 
\end{lemma}

The variable $\gamma$ is just coming along for the ride,  it is not essential in the proof of the lemma. However,  we will use this more general statement in Proposition \ref{prop:Dir_Spec_weight} below. 

\begin{proof}  By definition, multiplication by a constant $\kappa$ commutes with biasing, so it suffices to show the claim for $\kappa=1$. 

Let $A_\eps=[-\eps,\eps]$ and $\eps\to 0$. It suffices to show that the joint law of $(\lambda,\gamma, \{U\}_{2\pi})$ biased by $\Xi_U(A_\eps)$ converges to the law of $(\lambda,\gamma, V)$ biased by $f(V)$. 
Now let
$\varphi$ be a bounded continuous test function. We will show that as $\eps\to 0$
\begin{equation}\label{eq:a-cont}
\frac1{E[\Xi_U[A_\eps]]}
{E[\varphi(\lambda, \gamma,\{U\}_{2\pi}) \Xi_U[A_\eps]]}\to \frac{1}{E[f(V)]}E[\varphi(\lambda, \gamma, V) f(V)].
\end{equation}
We have 
\begin{align*}
{E[\varphi(\lambda, \gamma,\{U\}_{2\pi}) \Xi_U[A_\eps]]}&=
E[\varphi(\lambda, \gamma,\{U\}_{2\pi}) \sum_{w\in 2\pi \Z+U} \lambda'(w) 1_{A_\eps}(\lambda(w))]\\
&=E\int_{\R} \varphi(\lambda, \gamma,\{ u \}_{2\pi}) \lambda'(u) 1_{A_\eps}(\lambda(u)) f(\{u\}_{2\pi}) du.
\end{align*}
By a change of variables, this equals
\begin{align*}
E\int_{-\eps}^{\eps} \varphi(\lambda, \gamma,\{ \lambda^{-1}(x) \}_{2\pi})  f(\{\lambda^{-1}(x)\}_{2\pi}) dx
&=2\eps E \left[\varphi(\lambda, \gamma,\{ \lambda^{-1}(X_\eps) \}_{2\pi})  f(\{\lambda^{-1}(X_\eps)\}_{2\pi}) \right]
\end{align*}
where $X_\eps$ is uniform on $(-\eps,\eps)$ and independent of $\lambda$.


As $\eps\to 0$ we have $\{\lambda^{-1}(X_\eps)\}_{2\pi}\to V$ in probability. By the bounded convergence theorem and the assumptions on $f$ we get 
\[
\frac{1}{2\eps} E[\varphi(\lambda, \gamma,\{U\}_{2\pi}) \Xi_U[A_\eps]]\to E \left[\varphi(\lambda, \gamma,V)  f(V) \right].
\]
This convergence (together with the case $\varphi=1$) now 
implies \eqref{eq:a-cont}, and the statement of the lemma follows.
\end{proof}

\subsection*{Biasing and Dirac operators}

\begin{proposition}\label{prop:Dir_Spec_weight}
Suppose that $\cI$ is closed on the right, and  $R$, $\uu_0=[1,0]^t$ satisfy Assumptions 1 and 2 on $\cI$. For an $r\in \R\cup\{\infty\}$ we denote by $\mu_{R,r}$  the right spectral measure of the  Dirac operator $\btau_r=\Dirop(R_{\cdot}, \infty, r)$ on $\cI$.

Let $q$ be a random variable on $\R\cup\{\infty\}$ that is independent of $R$, and for which the law of the angle $2\operatorname{arccot}(q)\in \mathbb S$ has bounded density which is continuous and positive at $0$. 
Then $(\mu_{R,q},(R,q))$ biased by the weight of 0 has law given by $(\mu_{R,\infty},(R,\infty))$.
\end{proposition}

In words, biasing the operator $\btau_q$ by the weight of zero in the right spectral measure turns the random right boundary condition $q$ into  $\infty$. 

\begin{proof}
Consider the phase function $\alpha$ for the operator $\btau_r$ introduced in \eqref{eq:alpha},  this does not depend on the right boundary condition $r$.  Let $\lambda:\R\to \R$ be the inverse of the function $\alpha(1,\cdot)$, note that $\lambda^{-1}(0)=0$.

By Lemma \ref{lem:phase_spectral1} the spectral measure $\mu_{R, q}$ is defined as \eqref{def:Xi}
with $\kappa=2$ and $u=-2\operatorname{arccot}(q)\in (-\pi,\pi]$.
 We can now apply Lemma  \ref{lem:pullback} with $\gamma=(R,q)$. Using Lemma \ref{lem:phase_spectral1} again we identify $\Xi_V=\Xi_0$ as $\mu_{R,\infty}$.
\end{proof}

Proposition \ref{prop:Dir_Spec_weight} applied to $\btau_\beta$ leads to the following corollary. 

\begin{corollary}\label{cor:Sine_biased}
Consider $R, q$ and $\btau_\beta$ introduced in \eqref{eq:xy}-\eqref{tau_0}. For $r\in \R$ let  $\btau_{\beta,r}=\Dirop(R_{\cdot},\infty,r)$, and let  $\mu=\mu_{\textup{right},\btau_\beta}$ and $\mu_r=\mu_{\textup{right},\btau_{\beta,r}}$. Then  
the  $(\mu, (R,q))$ biased by the weight of 0 in $\mu$ has the same distribution as  $(\mu_\infty, (R,\infty))$.  In particular, $\mu$ biased by the weight of zero is $\mu_\infty$, the right spectral measure of $\Dirop(R_{\cdot},\infty,\infty)$.
\end{corollary}

\section{Finite support measures on the unit circle}
\label{sec:finite}

In \cite{BVBV_sbo} it was shown that one can associate a Dirac operator to a finitely supported probability measure on the unit circle. The construction relies on the theory of orthogonal polynomials \cite{OPUC}, see \cite{SimonOPUC1foot} for a shorter summary. 

Let $\nu$ be a probability measure whose support is exactly $n$ points $e^{i\lambda_j}, 1\le j \le n$ on the unit circle $\partial \UU$. Denote the Gram-Schmidt ortogonalization of the polynomials $1,z, \cdots, z^n$ with respect to $\nu$  by $\Phi_k, k=0, \dots, n$. These are the orthogonal polynomials with respect to $\nu$, with $\Phi_n(z)=\prod_{j=1}^n(z-e^{i \lambda_j})$. Together with the reversed polynomials  $\Phi_k^*(z)=z^k \overline{\Phi_k(1/\bar z)}$, they satisfy the famous Szeg\H o recursion (see e.g.~Section 1.5, volume 1 of \cite{OPUC}):
\begin{align}\label{eq:Szego1}
\binom{\Phi_{k+1}}{\Phi_{k+1}^*}=
\mat{1}{-\bar \alpha_k}{-\alpha_k}{1} \mat{z}{0}{0}{1} \binom{\Phi_{k}}{\Phi_{k}^*}, \qquad \binom{\Phi_0}{\Phi_0^*}
=\binom{1}{1},\qquad 0\le k\le n-1.
\end{align}
Here $\alpha_k$, $0\le k\le n-1$ are the Verblunsky coefficients, they satisfy $|\alpha_k|<1$ for $0\le k\le n-2$ and $|\alpha_{n-1}|=1$. The map between the probability measures supported on $n$ points on $\partial \UU$ and the  Verblunsky coefficients $\alpha_0, \dots, \alpha_{n-1}$ is invertible, and both the map and its inverse are continuous.

We  denote by $\mathcal{U}(z):=\frac{z-i}{z+i}$ the Cayley transform mapping  $\overline \HH$ to $ \overline \UU$, and introduce its matrix version
\begin{equation}\label{eq:U}
U=\mat{1}{-i}{1}{i}.
\end{equation}

The following definition constructs a path and a Dirac operator from a probability measure supported on finitely many points on $\partial \UU$.

\begin{definition}\label{def:path} 
Let $\nu$ be a probability measure supported on $n$ distinct points  on the unit circle, and let $\alpha_k, 0\le k\le n-1$ be its Verblunsky coefficients. Define  $b_k, $,  $0\le k\le n$ as 
\begin{align}\label{eq:discrete_path_b}
b_0&=0, \quad b_k=\mathcal P \mat{1}{\bar \alpha_0}{\alpha_0}{1}\cdots \mat{1}{\bar \alpha_{k-1}}{\alpha_{k-1}}{1} \binom{0}{1},
\qquad 1\le k\le n, 
\end{align}
and $z_k=\mathcal{U}^{-1} (b_k)$, $0\le k\le n$.

Let  $z(t)=z_{\lfloor nt\rfloor}$, $\uu_0=[1,0]^t$ and $\uu_1=[-z_n, -1]$ if $z_n \in \R$ and $\uu_1=[1,0]$ if $z_n=\infty$. We call $z_0, \dots, z_n$  the {\bf path parameter} of $\nu$ in $\HH$, and $\btau=\Dirop(z(\cdot), \uu_0, \uu_1)$ the {\bf Dirac operator corresponding to $\nu$}. We call $b_0, \dots, b_n$ the path parameter of $\nu$ in $\UU$. 
\end{definition}

The following result is proved in Proposition 16 of \cite{BVBV_sbo}.

\begin{proposition}
\label{prop:discrete_op_1}
Let $\nu$ be a probability measure supported on $n$ distinct points $e^{i \lambda_j}, 1\le j\le n$ on the unit circle $\{|z|=1\}$. Let $z_k, 0\le k\le n$ be the path parameter of $\nu$ in $\HH$ introduced in Definition \ref{def:path}, and let $\btau=\Dirop(z(\cdot), \uu_0, \uu_1)$ be the Dirac operator corresponding to $\nu$. 

Then $\btau$ satisfies Assumptions 1 and 2 on $\cI=[0,1]$, and also Assumption 3 if $z_n\neq \infty$. Moreover, the spectrum of $\btau$ is the set $\{n \lambda_j+2\pi n k: 1\le j\le n, k\in \ZZ\}$. The spectrum contains 0 if and only if $z_n=\infty$. 
\end{proposition}

The path parameter of  a probability measure supported on $n$ distinct points can be expressed in a more convenient way using the \emph{modified} Verblunsky coefficients $\gamma_k, 0\le k\le n-1$. These are defined 
recursively in terms of the Verblunsky coefficients as follows:
\begin{align}
    \gamma_0=\bar \alpha_0, \qquad \gamma_k=\bar \alpha_k \prod_{j=0}^{k-1} \frac{1-\bar \gamma_j}{1-\gamma_j}   , \quad 1\le k\le n-1. \label{eq:gamma_alpha}
\end{align}
By definition  we have $|\gamma_k|=|\alpha_k|$ for all $0\le k\le n-1$, in particular $|\gamma_{n-1}|=1$. The recursion \eqref{eq:gamma_alpha} shows that the map between the regular and modified Verblunsky coefficients is invertible. 

The modified Verblunsky coefficients are connected to the  affine group of isometries of the hyperbolic plane fixing a particular boundary point. In the half plane representation $\HH$ we choose the boundary point to be $\infty$: in that case the transformations are of the form $z\to \frac{1}{y}(z-x)$ with $x+i y\in \HH$. The matrix representation and the transformation corresponding to a $x+i y\in \HH$ is given by 
\begin{equation}\label{e:A1def}
A_{x+i y, \HH}= 
\mat{1}{-x}{0}{y}, \qquad \mathcal A_{x+i y, \HH}(z)= \mathcal P A_{x+i y, \HH} \binom{z}{1}.
\end{equation}
Note that $x+i y$ is the pre-image of $i$ under $\mathcal A_{x+i y, \HH}$.

In the unit disk $\UU$ representation of the hyperbolic plane we choose the fixed point to be $1$, in this case the corresponding linear fractional transformations can be parameterized by the pre-image $\gamma \in \UU$ of $0$. 
The matrix representation and the linear fractional transformation corresponding to $\gamma \in \UU$ are given by
\begin{equation}\label{eq:Adef}
A_{\gamma, \UU}= 
\mat{\frac{1}{1-\gamma}}
    {\frac{\gamma}{\gamma-1}}
    {\frac{\bar \gamma}{\bar \gamma-1}}
    {\frac{1}{1-\bar \gamma}}, \qquad \mathcal A_{\gamma, \UU}(z)= \mathcal P A_{\gamma,\UU} \binom{z}{1}.
\end{equation}
The half-plane and disk representations can be connected via the Cayley transform $\mathcal U$ and the corresponding matrix $U$:
\begin{align*}
A_{U^{-1}(\gamma),\HH}=U^{-1} A_{\gamma, \UU}\, U, \qquad \mathcal A_{U^{-1}(\gamma), \UU}=\mathcal U^{-1} \circ \mathcal A_{\gamma, \UU} \circ \mathcal U.
\end{align*}
For a particular sequence of modified Verblunsky coefficients  $\gamma_k, 0\le k\le n-1$ we introduce the variables $v_k, w_k\in \R\cup \{\infty\}$, $0\le k\le n-1$ as
\begin{align}\label{def:v_w}
    v_k+i w_k=\mathcal U^{-1}(\gamma_k)-i. 
\end{align}
Note that $v_k, w_k$ are finite for $0\le k \le n-2$ with $w_k>-1$. We have $w_{n-1}=-1$, with $v_{n-1}=\infty$ for $\gamma_{n-1}=1$, and $v_{n-1}\in \R$ otherwise. A direct computation gives
\begin{align}\label{eq:v_w}
w_k=2\Re \frac{\gamma_k}{1-\gamma_k} , \quad v_k=-2 \Im \frac{\gamma_k}{1-\gamma_k}, \quad \text{if $\gamma_k\neq 1$.}
\end{align}
The following lemma shows that the path parameter can be expressed in a simple way in terms of the modified Verblunsky coefficients. 
\begin{proposition}
Let $\nu$ be a probability measure supported on $n$ distinct points $e^{i \lambda_j}, 1\le j\le n$ on the unit circle $\partial \UU$. Let 
$b_k, z_k, 0\le k\le n$ be the path parameters of $\nu$ introduced in Definition \ref{def:path}, and let $\gamma_k, 0\le k\le n-1$ be the modified Verblunsky coefficients, with $w_k, v_k$ defined in \eqref{def:v_w}. Then following identities hold for $0\le k\le n-1$:
\begin{align}\label{eq:z_rec}
z_{k+1}&=z_k+(v_k+i w_k) \Im z_k, \\
b_{k+1}&=\frac{b_k+\gamma_k\frac{1-b_k}{1-\bar b_k}}{1+\bar b_k \gamma_k \frac{
1-b_k}{1-\bar b_k}}\label{eq:b_rec}.
\end{align}
Moreover, we have
\begin{equation}
\label{eq:bk_cA}
b_k=\cA_{\gamma_0,\UU}^{-1}\circ \cdots\circ  \cA_{\gamma_{k-1},\UU}^{-1}(0),
\end{equation}
with $\cA_{\gamma_{n-1},\UU}^{-1}(0)$ defined as the limit of $\cA_{\gamma,\UU}^{-1}(0)$ as $\gamma\to \gamma_{n-1}$ with $\gamma\in \UU$.
\end{proposition}
\begin{proof}
From \eqref{eq:Szego1} and \eqref{eq:gamma_alpha} it follows by induction that 
\begin{align}\label{eq:Phi_1}
\Phi_k(1)=\prod_{j=0}^{k-1} (1-\gamma_j), \qquad \Phi_k^*(1)=\prod_{j=0}^{k-1} (1-\bar \gamma_j).
\end{align}
Using this identity with \eqref{eq:gamma_alpha} again, we obtain 
\[
\mat{1}{-\bar \alpha_k}{-\alpha_k}{1}=\mat{\Phi_{k+1}(1)}{0}{0}{\Phi_{k+1}^*(1)} A_{\gamma_k,\UU} \mat{\Phi_{k}(1)}{0}{0}{\Phi_{k}^*(1)}^{-1}.
\]
If $|\alpha|<1$ then 
\[
\mat{1}{-\bar \alpha}{-\alpha}{1}^{-1}=\frac{1}{1-|\alpha|^2} \mat{1}{\bar \alpha}{\alpha}{1}.
\]
Hence from \eqref{eq:discrete_path_b} we get
\begin{align*}
b_k=\mathcal P \mat{1}{-\bar \alpha_0}{-\alpha_0}{1}^{-1}\cdots \mat{1}{-\bar \alpha_{k-1}}{-\alpha_{k-1}}{1}^{-1} \binom{0}{1}=\mathcal P A_{\gamma_0, \UU}^{-1} \cdots A_{\gamma_{k-1}, \UU}^{-1} \binom{0}{1},
\end{align*}
proving \eqref{eq:bk_cA}. Since $A_{\gamma, \UU}^{-1}(0)=\gamma$, for $1\le k\le n-1$ we also get the identity
\begin{align}
    A_{b_k, \UU}=A_{\gamma_{k-1}, \UU} \cdots A_{\gamma_0, \UU}=A_{\gamma_{k-1}, \UU} A_{b_{k-1}, \UU}. \label{eq:Ab_rec}
\end{align}
Since $z_k=\mathcal U^{-1}(b_k)$ and $v_k+i(w_k+1)=\mathcal U^{-1}(\gamma_k)$ equation \eqref{eq:Ab_rec} implies
\begin{align}\label{eq:Az_rec}
A_{z_k, \HH}=A_{v_{k-1}+i(w_{k-1}+1), \HH} \, A_{z_{k-1}, \HH}. 
\end{align}
From \eqref{eq:Ab_rec} and \eqref{eq:Az_rec} we obtain \eqref{eq:b_rec} and \eqref{eq:z_rec} for $0\le k\le n-2$. The identity for $k=n-1$ in both cases follow by taking the appropriate limits.
\end{proof}

The following statement extends the results of Proposition \ref{prop:discrete_op_1} to show that the left spectral measure of the Dirac operator associated to a finitely supported probability measure $\nu$  is a scaled and periodically extended version  of $\nu$.

\begin{proposition}\label{prop:Dir_discrete}
Let $\nu$ be a probability measure supported on $n$ distinct points $e^{i \lambda_j}, 1\le j\le n$ on $\partial \UU$, and let $\btau$ be the Dirac operator corresponding to $\nu$ as defined in Definition \ref{def:path}.
Then
\begin{align}\label{eq:disc_Dir_spec}
    \mu_{\textup{left}, \btau}(n \lambda_j+2n k \pi)=2 n\, \nu(e^{i \lambda_j}), \qquad k\in \ZZ, 1\le j\le n.
\end{align}
\end{proposition}
\begin{proof}
For $\lambda\in \R$ let  $H(t,\lambda)$  be the solution of the eigenfunction equation
\begin{align}\label{eq:ODE_H}
\btau H=\lambda H, \qquad H(0,\lambda)=\uu_0=\bin{1}{0}.
\end{align}
Denote by $z_k, 0\le k\le n$ the path parameter of $\nu$ in $\HH$, and introduce 
\[
X_k=A_{z_k, \HH}.
\]
From Definition \ref{def:path} it follows that $\btau=R^{-1} J \frac{d}{dt}$ where
\begin{align}\label{eq:discrete_R}
    R(t)=R_k=\frac{X_k^t X_k}{2 \det X_k}, \qquad \text{for $t\in [k/n,(k+1)/n)$}.
\end{align}
Introduce the normalized orthogonal polynomials 
 \begin{align}\label{def:psi_X}
    \Psi_k(z)=\Phi_k(z)/\Phi_k(1), \quad \Psi_k^*(z)=\Phi_k^*(z)/\Phi_k^*(1).
\end{align}
From Proposition 32 of \cite{BVBV_szeta} it follows that 
the function $H(t,\lambda)$ satisfies the identities
\begin{align}\label{eq:discrete_H_1}
H(\tfrac{k}{n},\lambda)&=H_k(\lambda)=e^{-\frac{i\lambda k}{2n}}X_k^{-1}U^{-1} \binom{\Psi_k(e^{i \lambda/n})}{\Psi_k^*(e^{i \lambda/n})}, \qquad 0\le k\le n-1,\\ \label{eq:discrete_H_2}
H(t,\lambda)&=X_{k}^{-1} U^{-1} \mat{e^{\frac{i\lambda}{2}(t-k/n))}}{0}{0}{e^{\frac{-i\lambda }{2}(t-k/n)}}  U X_k\, H_k(\lambda), \qquad t\in (k/n,(k+1)/n].
\end{align}
From \eqref{eq:discrete_R},  \eqref{eq:discrete_H_1} and \eqref{eq:discrete_H_2}  we obtain
\[
\|H(\cdot, \lambda)\|_R^2=\frac1{n} \sum_{k=0}^{n-1} H_k(\lambda)^t R_k H_k(\lambda)=\frac1{n} \sum_{k=0}^{n-1} \frac{1}{2 \Im z_k} |\Psi_k(e^{i \lambda/n})|^2,  
\]
where in the last step we also used that $|\Phi_k(e^{i \lambda/n})|=|\Phi_k^*(e^{i \lambda/n})|$ for $\lambda\in \R$.

Setting $\tilde \lambda=n \lambda_j+2n k \pi$ for a $k\in \ZZ$, $1\le j\le n$ we get 
\begin{align}\label{eq:disc_norm_sq}
 \|H(\cdot, \tilde \lambda)\|_R^2=\frac1{n} \sum_{k=0}^{n-1} \frac{1}{2\Im z_k} |\Psi_k(e^{i \lambda_j})|^2,     
\end{align}
and since $|H(0,\tilde \lambda)|=|\binom{1}{0}|=1$, this leads to
\begin{align}
     \mu_{\textup{left}, \btau}(n \lambda_j+2n k \pi)= \frac{2n}{\sum_{k=0}^{n-1} \frac{1}{\Im z_k} |\Psi_k(e^{i \lambda_j})|^2}.
\end{align}
Hence to prove \eqref{eq:disc_Dir_spec} we need to show that \begin{align}\label{eq:intermediate}
    \nu(e^{i \lambda_j})^{-1}=\sum_{k=0}^{n-1} \frac{1}{\Im z_k} |\Psi_k(e^{i \lambda_j})|^2.
\end{align}
Fix $1\le j\le n$, and let $g$ be the degree $n$ polynomial with $g(e^{i \lambda_j})=1$ and $g(e^{i \lambda_\ell})=0$ for $1\le \ell\le n, \ell\neq j$. The polynomials $\Psi_k, 0\le k\le n-1$ are orthogonal with respect to $\nu$, and by construction $\int \bar \Psi_k(z) g(z) d\nu(z)=\bar \Psi_k(e^{i \lambda_j}) \nu(e^{i \lambda_j})$. Hence we get
\[
g(z)=\sum_{k=0}^{n-1}\Psi_k(z)  \frac{\bar \Psi_k(e^{i \lambda_j}) \nu (e^{i \lambda_j})}{\|\Psi_k\|_\nu^2},
\]
and using $z=e^{i \lambda_j}$ we obtain
\begin{align}\label{eq:disc_spec_2}
   1=\nu(e^{i \lambda_j}) \sum_{k=0}^{n-1} \frac{|\Psi_k(e^{i \lambda_j})|^2}{\|\Psi_k\|^2_\nu} .
\end{align}
The needed identity \eqref{eq:intermediate} follows  once we show that $\Im z_k=\|\Psi_k\|_\nu^2$.

By Section 1.5 of volume 1, \cite{OPUC}  we have
\[
\|\Phi_k\|_\nu^2=\prod_{\ell=0}^{k-1}(1-|\gamma_\ell|^2), 
\]
which together with \eqref{eq:Phi_1} gives $\|\Psi_k\|^2_\nu=\prod_{\ell=0}^{k-1} \frac{1-|\gamma_\ell|^2}{|1-\gamma_\ell|^2}$. On the other hand, from \eqref{eq:v_w} and \eqref{eq:z_rec} we get
\[
\Im z_k=\prod_{\ell=0}^{k-1}(1+2 \Re \frac{\gamma_\ell}{1-\gamma_\ell})=\prod_{\ell=0}^{k-1} \frac{1-|\gamma_\ell|^2}{|1-\gamma_\ell|^2}, \quad 0\le k\le n-1,
\]
which finishes the proof.
\end{proof}

\begin{remark}
As a simple example we can fix $\theta\in (-\pi,\pi)$ and consider the uniform probability measure $\nu$ on the points $e^{i \frac{\theta+2k\pi}{n}}$, $0\le k\le n-1$. In this case 
\begin{align*}
  \Phi_k(z)=z^k \quad\text{ for $0\le k\le n-1$, and }\qquad   \Phi_n(z)=z^n-e^{i \theta}.
\end{align*}
The Verblunsky coefficients are
\begin{align*}
 \alpha_k=0 \quad \text{  for $0\le k\le n-2$, and} \qquad \alpha_{n-1}=e^{-i \theta}.
\end{align*}
The  modified Verblunski coefficients are \begin{align*}
 \gamma_k=0\quad \text{ for $0\le k\le n-2$, and }\qquad \gamma_{n-1}=e^{i \theta}, 
\end{align*}
the  path parameter in $\HH$ is just $z_k=i$ for $0\le k\le n-1$ with $z_n=\cot(\theta/2)$, and the corresponding function $R(t)$ is $\tfrac12 I$. 

The Dirac operator is $2 J \tfrac{d}{dt}$ with boundary conditions $\uu_0=[1,0]^t$ and $\uu_1=[-\cot(\theta/2),-1]^t$. The eigenvalues are of the form $2\pi k+\theta, k\in \ZZ$, and the $L^2_R$-normalized eigenfunction corresponding to $\lambda$ is $\sqrt{2} [\cos(\lambda t/2), \sin(\lambda t/2)]^t$. Therefore the weights of the spectral measure are equal to 2 at the eigenvalues for both $\mu_{\text{left}}$ and $\mu_{\text{right}}$, illustrating the identity \eqref{eq:disc_Dir_spec}.
\end{remark}

\section{Finite approximations of the spectral measure of $\btau_\beta$}

The goal of this section is to prove Proposition \ref{prop:sine_weights} identifying the right spectral weights of the $\btau_\beta$ operator. Recall the definition of the circular beta-ensemble and the Killip-Nenciu measure.

\begin{definition}\label{def:circ_KN}
The size $n$ circular beta ensemble with parameter $\beta>0$ is a distribution of $n$ points $e^{i\theta_1}, \dots, e^{i \theta_n}$ on $\partial \UU$ with joint density function
\begin{align}
    \frac{1}{Z_{n,\beta}} \prod_{1\le j\le k\le n} \left|e^{i \theta_j}-e^{i \theta_k}\right|^2.
\end{align}
The size $n$ Killip-Nenciu measure with parameter $\beta>0$ is a random probability measure $\mu_{n,\beta}^{\textup{KN}}$ with support given by the size $n$ circular beta ensemble, with weights distributed according to the Dirichlet$(\beta/2,\dots, \beta/2)$ distribution, independently of the support. 
\end{definition}

\begin{theorem}[\cite{KillipNenciu},  \cite{BNR2009}]\label{thm:KillipNenciu} The Verblunski coefficients $\alpha_0, \dots, \alpha_{n-1}$ of the random probability measure $\mu_{n,\beta}^{\textup{KN}}$ are independent, rotationally invariant, and for $0\le k\le n-2$ we have $|\alpha_k|^2\sim$ {Beta}$(1,\tfrac{\beta}{2}(n-k-1))$. 
The joint distribution of the modified Verblunski coefficients $\gamma_0, \dots, \gamma_{n-1}$ is the same as the joint distribution of $\alpha_0, \dots, \alpha_{n-1}$
\end{theorem}

\begin{definition}
We denote by $\Circop$ be the random Dirac operator associated to the Killip-Nenciu measure $\mu_{n,\beta}^{\textup{KN}}$ via Definition \ref{def:path}.
\end{definition}

\begin{definition}
Let $b_1, b_2$ be independent standard Brownian motion  and for $t\in [0,1)$ set $v=v(t)=\frac{4}{\beta} \log(1-t)$. Let $\tilde z(t)=\tilde x(t)+i \tilde y(t)$ with $\tilde y(t)=e^{b_2(v)-v/2}$ and $\tilde x(t)=\int_0^v e^{b_2(s)-\frac{s}{2}} db_2$. Let $\tilde z(1)=\lim_{t\to 1} \tilde z(t)$, and set $\Sineop=\Dirop(\tilde z,\infty,\tilde z(1))$.
\end{definition}
The operator $\Sineop$ was introduced in \cite{BVBV_sbo}, in \cite{BVBV_szeta} it was shown to be orthogonally equivalent to the operator $\btau_\beta$  introduced in \eqref{tau_0}. 
Recall $\rho, S$ and $\mathfrak{r}$ introduced before Lemma \ref{lem:reversal}.


\begin{proposition}[Proposition 44, \cite{BVBV_szeta}]\label{prop:tau_sine} Let  $\btau_\beta=\Dirop(x+i y,\uu_0, \uu_1)$ defined via \eqref{eq:xy}-\eqref{tau_0}, and set 
$z(t)=x(t)+i y(t)$. Let $T_q$ and  $\mathcal{T}_q$ defined via \eqref{def:hyprot}. Then  the operator 
\[\tilde \btau=\rho^{-1} (S  T_q)\btau_\beta (S  T_q)^{-1} \rho
\]
is orthogonal equivalent to $\btau_\beta$, and has the same distribution as $\Sineop$. We have $\tilde \btau=\Dirop(\tilde z,\infty,\tilde z(1))$ where  $\tilde z=\rho \mathfrak{r}\mathcal{T}_q z$, in particular $\tilde z(1)=-q$.
\end{proposition}
The following statement follows  from Proposition \ref{prop:tau_sine} and Lemma \ref{lem:isometry}.

\begin{corollary}
The right (left) spectral measure of $\btau_\beta$ has the same distribution as the left (right) spectral measure of $\Sineop$. 
\end{corollary}

\cite{BVBV_op} provided a coupling of the $\Circop$ and $\Sineop$ operators and their driving paths under which $\|\res \Circop-\res \Sineop\|_{\textup{HS}}\to 0$ a.s. We summarize some of the properties of this coupling in the proposition below. Let 
$$
d_{\HH}(z_1,z_2)=\operatorname{arccosh}\left(1+ \frac{|z_1-z_2|^2}{2 \Im z_1 \Im z_2}\right)
$$
denote the hyperbolic distance between two points $z_1, z_2$ of $\HH$.

\begin{proposition}[\cite{BVBV_op}]\label{prop:coupling}
There is a coupling of the operators $\Circop=\Dirop(\tilde z_n, \infty, \tilde z_{n}(1))$ and $\Sineop=\Dirop(\tilde z, \infty, \tilde z(1))$ with the following properties. The driving paths $\tilde z_n, \tilde z$ satisfy $\tilde z_n(1)=\tilde z(1)$,
and there is a random  $N_0$ so that for all $n\ge N_0$  the following uniform bounds hold:
 \begin{align}\label{eq:coupling_1}
     d_{\HH}(\tilde z_n(t), \tilde z(t))\le& \frac{\log^{3-1/8}n}{\sqrt{(1-t)n}}, 
    \qquad 0\le t\le t_n=1-\tfrac{1}{n}\log^6 n,\\
    d_{\HH}(\tilde z_n(t_n), \tilde z(t))\le& \frac{144}{\beta} (\log \log n)^2, \qquad t_n\le t<1.\label{eq:coupling_2}
 \end{align}
\end{proposition}

We know have all the ingredients to prove Proposition \ref{prop:sine_weights}.

\begin{proof}[Proof of Proposition \ref{prop:sine_weights}]
Consider the coupling of Proposition 
\ref{prop:coupling}, and denote the left spectral measure of $\Circop$ by $\mu_{\text{left},n}$. 
By Theorem \ref{thm:KillipNenciu} and Proposition \ref{prop:Dir_discrete}   we get that the support of $\mu_{\text{left},n}$ is given by $n \Lambda_n+n 2\pi n\Z$ where $\Lambda_n$ is distributed as the size $n$ circular beta ensemble. The main result of \cite{BVBV_op} is that in the coupling of Proposition 
\ref{prop:coupling} the set $n \Lambda_n+n 2\pi n\Z$ converges pointwise to the spectrum of $\Sineop$, which has the same distribution as the spectrum of $\btau_\beta$. 

The weights of $\mu_{\text{left},n}$ are given by the periodic extension of the values $2n (X_{1,n}, \dots, X_{n,n})$ where $(X_{1,n}, \dots, X_{n,n})$ has Dirichlet$(\beta/2, \dots, \beta/2)$ distribution and independent of $\Lambda_n$. The Dirichlet$(\beta/2, \dots, \beta/2)$ distribution can be obtained by normalizing $n$ independent Gamma$(\beta/2)$ random variables with their sum. Using the strong law of large numbers it now follows that the weights of $\mu_{\text{left},n}$ converge to an i.i.d.~sequence of Gamma random variables with shape parameter $\beta/2$ and expected value 2. If we show that
\begin{align}\label{eq:conv}
    \mu_{\text{left},n}\to \mu_{\text{left},\Sineop}\quad \text{a.s.~in the vague topology of measures}
\end{align}
then the statement of the proposition follows from Proposition \ref{prop:tau_sine}. To prove \eqref{eq:conv} we will use Lemma \ref{lem:spec_convergence}, but first we need to transform the operators so that the domain is an interval that is closed on the right. 

Note  that the operators $\Circop$ and $\Sineop$ share the same right condition $\tilde z(1)$. Consider the transformations $\rho$ and $\mathfrak{r}$, and the matrix $S$ introduced before Lemma \ref{lem:reversal}. Let $T_*=T_{-\tilde z(1)}$ defined via \eqref{def:hyprot}. Then by Proposition \ref{prop:tau_sine} the operator $\rho (S T_*)^{-1} \Sineop (S T_*) \rho$ has the same distribution as $\btau_\beta$, and with a minor abuse of notation we will use the notation $\btau_\beta$ for it. The driving path $z$ of $\btau_\beta$ satisfies $z=\rho \mathfrak{r} \mathcal{T}_*^{-1} \tilde z$.

Introduce the similarly transformed versions of the $\Circop$ operators
\[
\btau_{\beta,n}=\rho (S T_*)^{-1} \Circop (S T_*) \rho.
\]
These operators have driving paths $z_n=\rho \mathfrak{r} \mathcal{T}_*^{-1} \tilde z_n$.

We will show that the operators $\btau_\beta$ and $\btau_{\beta,n}$ satisfy the conditions of Lemma \ref{lem:spec_convergence}. Since 
\[
\mu_{\text{left},n}=\mu_{\text{right},\btau_{\beta,n}}, \qquad  \mu_{\text{left},\Sineop}=\mu_{\text{right},\btau_{\beta}},
\]
this  implies \eqref{eq:conv} and the proposition.

In \cite{BVBV_szeta} it was shown that in the coupling of Proposition \ref{prop:coupling} one has
$$\|\res \btau_n -\res \btau\|_{\textup{HS}}\to 0, \qquad  \ttr_{\btau_n}\to \ttr_{\btau},$$
with probability one. (See Theorem 47 and Proposition 48 and their proofs in \cite{BVBV_szeta}.)  Hence  we only need to show that in our coupling
\begin{align}\label{eq:a0_goal}
\int_0^1 |a_0(s)-a_{0,n}(s)|^2ds\to 0.
\end{align}
Note that in $\btau$ and $\btau_n$ the initial condition is $\uu_0=[1,0]$, hence by \eqref{eq:ac} we have
\[
|a_0(s)-a_{0,n}(s)|^2=\left|\frac{1}{\sqrt{\Im z(s)}}-\frac1{\sqrt{\Im z_n(s)}} \right|^2.
\]
An explicit computation shows that 
\[
\left|\frac{1}{\sqrt{y_1}}-\frac1{\sqrt{y_2}} \right|^2\le \frac{4}{y_1} \sinh(\frac12 d_{\HH}(y_1,y_2))^2 
\]
Since $\Im z(t), t\in (0,1]$ is distributed as $y_t$ from \eqref{eq:xy}, for any $\eps>0$ we can find a random $C<\infty$ so that 
\begin{align}\label{eq:a0_1}
|a_0(s)|^2=(\Im z(t))^{-1}\le C t^{2/\beta-\eps}.
\end{align}
From the definition of $z, z_n$ we also get 
\[
d_{\HH}(z(s),z_n(s))=d_{\HH}(\tilde z(1-s),\tilde z_n(1-s))
\]
We can now use the coupling bounds of Proposition \ref{prop:coupling} to get the following estimates (assuming that $n$ is sufficiently large). If $1-t_n\le s\le 1$ then
\begin{align}\label{eq:a0_2}
|a_0(s)-a_{0,n}(s)|^2\le  c \cdot C s^{2/\beta-\eps} \left(\frac{\log^{3-1/8}n}{\sqrt{s n}}\right)^2 
\end{align}
with an absolute constant $c$. If $0<s<1-t_n$ then
\begin{align}\label{eq:a0_3}
|a_0(1-t_n)-a_{0,n}(s)|^2\le  C' s^{2/\beta-\eps} e^{c (\log \log n)^2}
\end{align}
again with an absolute constant $c$. Choosing $0<\eps<2/\beta$ and using \eqref{eq:a0_1}-\eqref{eq:a0_3} we obtain \eqref{eq:a0_goal}, which finishes the proof. 
\end{proof}

We can now complete the proof of Theorem \ref{thm:Sine_Palm}.
\begin{proof}[Proof of Theorem \ref{thm:Sine_Palm}]
Consider $R, q$ and $\btau_\beta$ introduced in \eqref{eq:xy}-\eqref{tau_0}. For $r\in \R$ let  $\btau_{\beta,r}=\Dirop(R_{\cdot},\infty,r)$, and let  $\mu=\mu_{\textup{right},\btau_\beta}$ and $\mu_r=\mu_{\textup{right},\btau_{\beta,r}}$.
By Corollary \ref{cor:Sine_biased}, $\mu$ biased by the weight of zero is $\mu_\infty$, the right spectral measure of $\Dirop(R_{\cdot},\infty,\infty)$. It is known that the support of $\mu$ (the $\Sineb$ process) is translation invariant with intensity $\frac{1}{2\pi}$ (see \cite{BVBV}). By Proposition \ref{prop:sine_weights} the weights in $\mu$ are i.i.d.~with a finite expectation, and they are independent of the support. Hence $\mu$ satisfies the conditions of Proposition \ref{prop:Palm}, which means that $\mu_\infty$ is the Palm measure of $\mu$. Using again the fact that the weights of $\mu$ are i.i.d.~with finite expectation and they are independent of the support of $\mu$, together with Definition \ref{def:palm} we get the statement of Theorem \ref{thm:Sine_Palm}. 
\end{proof}

\section{Aleksandrov measures} \label{sec:Aleks}

Biasing a measure on the unit circle by the weight of one is closely related to the classical notion of an Aleksandrov measure. We review some basic facts about such measures, and refer the reader to  \cite{OPUC} volume 1, Section 3.2 and volume 2,  Section 10.2 for more background.

\begin{definition}\label{def:Aleks}
Let $\nu$ be a probability measure supported on $n$ distinct points on the unit circle. Denote its  Verblunsky coefficients by $\alpha_k, 0\le k\le n-1$. For $|\eta|=1$ the probability measure corresponding to the Verblunsky coefficients $\eta \alpha_k, 0\le k\le n-1$ is denoted by $\nu_\eta$, and it is called the {\bf Aleksandrov measure} corresponding to $\nu$ with Aleksandrov parameter $\eta$.
\end{definition}

\begin{lemma}\label{lem:Aleks_charge}
\begin{enumerate} The following statements hold for the Aleksandrov measure $\nu_\eta$.
    \item The path parameter $b_{\eta,0}, \dots, b_{\eta,n}$ of $\nu_\eta$ in $\UU$ is given by  $ \eta^{-1} b_0, \dots, \eta^{-1} b_n $.
    \item $\nu_{\eta}(1)>0$ if and only if $\eta=b_n$.
    \item The measure $\nu_\eta$ is continuous as a function of $\eta$.
\end{enumerate}
\end{lemma}
\begin{proof}
We have
\[
\mat{\eta}{0}{0}{1}^{-1} \mat{1}{\bar \alpha}{\alpha}{1}\mat{\eta}{0}{0}{1}=\mat{1}{\bar \eta \bar \alpha}{\eta \alpha}{1},
\]
hence the first statement follows by \eqref{eq:discrete_path_b}.

By part (b) of Proposition \ref{prop:Dir_discrete} it follows that $\nu_\eta(1)>0$ exactly if $b_{\eta,n}=1$. This shows that the first statement of our lemma implies the second one. 

The third statement follows from the fact that a probability measure supported on $n$ distinct points depends continuously on its Verblunsky coefficients (see Theorem 1.5.6 in  Volume 1 of \cite{OPUC}). 
\end{proof}

If one averages the  Aleksandrov measure in its  parameter then the result is the uniform measure on the unit circle. 
\begin{theorem}[Theorem 10.2.2, volume 2,  \cite{OPUC}]\label{thm:spectral_average}
For a finitely supported discrete probability measure $\nu$ on the unit circle the averaged measure $\frac{1}{2\pi} \int_0^{2\pi} \nu_{e^{i \varphi}} d\varphi$ is the uniform measure. 
\end{theorem}

The next theorem ties Aleksandrov measures to conditioning. See Section \ref{sec:biasing}
for the definition of conditioning by the weight of a point, and recall the path parameter $b_0,\ldots, b_n$ of a measure $\mu$ from Section \ref{sec:finite}.
\begin{proposition}\label{prop:Aleks_cont} Let $a_k$ be a sequence of points on the unit circle, let $\mu$ be a measure supported on $n$ points with $b_n$ defined in Definition \ref{def:path}, and let $\mu_{a_k}$ denote the Aleksandrov measure of $\mu$ with parameter $a_k$. Then the following are equivalent 
\begin{enumerate}
\item $a_k\to b_n$

\item $\mu_{a_k}\to \mu_{b_n}$

\item For every neighborhood $A$ of 1, $\mu_{a_k}(A)>0$ for all large enough $k$.  
\end{enumerate}
\end{proposition}
\begin{proof}
The measure $\mu$ is continuous as a function of its Verblunsky parameters, see Theorem 1.5.6 in volume 1, \cite{OPUC}, hence $\mu_a$ is continuous in $a$. This immediately shows that 1 implies 2. Since $\mu_{b_n}(1)>0$ by Lemma \ref{lem:Aleks_charge}, 2 implies 3.

Now assume 3, and take a converging subsequence of $a_{k_m}\to \eta$. The continuity of $\mu_x$ in $x$ implies that   $\mu_{a_k}\to \mu_{\eta}$, in particular the support of $\mu_{a_k}$ converges to the support of $\mu_{\eta}$. If $\eta\neq b_n$ then the support of $\mu_\eta$ does not contain 1, and we could find a neighborhood $A$ of 1 so that for large enough $k$ we have $\mu_{a_k}(A)=0$, contradicting 3. This shows that any subsequential limit of $a_k$ is equal to $b_n$, proving 1. 
\end{proof}

\begin{corollary}\label{cor:cond}
Fix $\mu$ supported on $n$ points on the unit circle, and let $U$ be any random variable so that $b_n\in \operatorname{supp} \operatorname{law} U$. Then $\mu_U$ conditioned on the weight of 1 is a deterministic measure equal to  $\mu_{b_n}$.
\end{corollary}
\begin{proof} 
For $0<\eps<1$ let $A_\eps$ be the arc of the unit circle corresponding to the angles in $(-\pi\eps,\pi\eps)$.

First we claim that $E\mu_U(A_{\eps})>0$ for all $\eps$. Let $\varphi$ be a continuous  function with values in $[0,1]$ with $\varphi(1)=1$ supported on $A_\eps$. If $u\to b_n$ then $\mu_u\to \mu_{b_n}$, so the functional $u\mapsto \int \varphi d\mu_{u}$ is continuous. Let $B$ the set of $u$   so that $\int \varphi d\mu_{u}>\int \varphi d\mu_{b_n}/2=:a/2$. Note that $a \ge \mu_{b_n}(1)>0$. By continuity, $B$ is an open neighborhood of $b_n$, so $P(U\in B)>0$, and on this event, 
$$
\mu_U(A_{\eps_k})>\int \varphi d\mu_{U}>a/2
$$
so $E\mu_U(A_{\eps_k})>a P(U\in B)/2>0$. 

Let $\mu_{U_k}$ be distributed as $\mu_U$ biased by $\mu_U(A_{\eps_k})$; this exists since $E\mu_U(A_{\eps_k})>0$. 

The sequence  $\mu_{U_k}$ satisfies Proposition \ref{prop:Aleks_cont} (3), hence by the Proposition, $U_k\to b_n$ and $\mu_{U_k}\to \mu_{b_n}$ in law. 
\end{proof}

\begin{corollary}[Conditioning on a point at 1.]\label{cor:cond_1}
Let $\mu$ be a random measure supported on $n$ points on the unit circle so that for every $|u|=1$ the Aleksandrov measure $\mu_u$ has the same distribution as $\mu$. Then $\mu$ conditioned on the weight of 1 has law $\mu_{b_n}$.
Moreover, if $\mu$ is rotationally invariant, then $\mu_{b_n}$ is the Palm measure of $\mu$.

\end{corollary}

\begin{proof}
Let $U$ be uniform. Consider the joint distribution $(X,\mu_U)$ where $X$ is a point on the unit circle with distribution given by $\mu_U$. Let $A_\eps$ be the event that $X$ is in the arc corresponding to $(-\pi \eps, \pi \eps)$

Take a bounded continuous functional $f$ on the space of probability measures times the unit circle. By the last claim of Corollary \ref{cor:cond}  we have
$$
Y_\eps:=\frac{E(f(\mu_U,U)\,;\,A_\eps\,|\,\mu)}
{P(A_\eps\,|\,\mu)}
\to f(\mu_{b_n},b_n)\\ \qquad \mbox{a.s.}
$$
Note that $|Y_\eps|\le \sup |f|$ a.s.  By spectral avereging, Theorem \ref{thm:spectral_average}, $P(A_\eps\,|\,\mu)=\eps$. Therefore 
$$
\eps^{-1}E(f(\mu_U,U)\,;\,A_\eps\,|\,\mu)\to f(\mu_{b_n},b_n)\\ \qquad \mbox{a.s.}
$$
Taking further expectations and using the bounded convergence theorem gives the required convergence  
\[
EY_\eps=\eps^{-1}E(f(\mu_U,U)\,;\,A_\eps)\to E f(\mu_{b_n},b_n).
\]
Thus we see that the law of $\mu_U$ given $A_\eps$ converges to the law of $\mu_{b_n}$. This limit is the definition of $\mu_U$ biased by the weight of 1. 

By Proposition \ref{prop:Palm} and Remark \ref{rem:circle}, if $\mu$ is rotationally invariant, then $\mu_{b_n}$ is the Palm measure of $\mu$.
\end{proof}

\section{Verblunski coefficients after biasing by the weight of one}

In this section we describe how random Verblunski coefficients change after conditioning on the weight of one. We have already seen that this is related to Aleksandrov measures, but the connection uses the path parameter $b_k, 0\le k\le n$, and the dependence can be complicated. In many cases, we still have simple formulas. For this, we again use the connection between the  modified Verblunski coefficients and the affine group of isometries of the hyperbolic plane introduced in Section \ref{sec:finite}.  

Recall that the affine group of linear fractional transformations of the unit disk fixing the point $1$ can be parametrized with the pre-image $\gamma$ of 0. The linear fractional transformation $\mathcal A_{\gamma,\UU}$ can be represented as a $2\times 2$ matrix $A_{\gamma, \UU}$ via \eqref{eq:Adef}.

The inverse of $A_{\gamma,\UU}$ corresponds to a linear fractional transformation whose parameter is denoted by $\gamma^\iota$:
\begin{align}\label{eq:iota}
     \gamma^\iota=-\gamma\frac{1-\bar \gamma}{1-\gamma},  \qquad A_{\gamma, \UU}^{-1}=A_{\gamma^\iota,\UU}.
\end{align}
Note that $\gamma\to \gamma^\iota$ is not an analytic function: it does not change the modulus, 
$
|\gamma|=|\gamma^\iota|,
$
but it is not a rotation around 0.
From the definition \eqref{eq:iota} we have
\begin{align}\label{eq:A_inverse}
\mathcal A_{\gamma, \UU}(\gamma)=0,\quad  \mbox{ i.e.  } ,\quad \mathcal A_{\gamma, \UU}^{-1}(0)=\mathcal A_{\gamma^\iota}(0)=\gamma.
\end{align}

Using  the $\gamma\to \gamma^{\iota}$ map the connection between the regular and modified Verblunski coefficients \eqref{eq:gamma_alpha} simplifies to 
\[
\gamma_0=\bar \alpha_0, \qquad -\frac{\gamma_k}{\gamma_{k-1}^\iota}=\frac{\bar \alpha_k}{\bar \alpha_{k-1}}, \quad 1\le k\le n-1.
\]
For $u\in \partial \UU,\gamma\in \UU$ define the Poisson kernel as
$$
\Poi(\gamma,u)=\Re \frac{u+\gamma}{u-\gamma}=\frac{1-|\gamma|^2}{|u-\gamma|^2}.
$$
The normalization is so that if $\Theta$ is uniform random on $\partial \UU$, then the expected value of $\Poi(\gamma,\Theta)$ is 1. We have
\begin{equation}
    \label{eq:Ay}
\Poi(\gamma,1)= \frac{1-|\gamma|^2}{|1-\gamma|^2}=\det A_{\gamma, \UU}=\mathcal{A}_{\gamma, \UU}'(1).
\end{equation}

Our next result shows how the $\gamma\to \gamma^{\iota}$ operation shows up naturally when considering the reversed (and rescaled) version of a path corresponding to a discrete probability measure on $\partial \UU$.   
\begin{lemma}\label{lem:reverse}
Let $\nu$ be a probability measure supported on $n$ distinct points on the unit circle. Denote the modified  Verblunski coefficients of $\nu$ by $\gamma_k, 0\le k\le n-1$, and consider the path $b_k, 0\le k\le n$ defined by \eqref{eq:discrete_path_b} in Definition \ref{def:path}. 
For $0\le k\le n-1$ define the reversed path parameter as 
\begin{align}
b'_k=
\cA_{b_{n-1}, \UU}(b_{n-k-1}) \qquad 0\le k\le n-1.
\end{align}
Then $b_0', \dots, b_{n-1}'$ are the first $n$ elements of the path produced by a sequence of modified Verblunsky coefficients starting with $\gamma^\iota_{n-2},\ldots, \gamma^\iota_0$. 
\end{lemma}
\begin{proof}
From  \eqref{eq:Ab_rec} we have 
\[
\cA_{b_k, \UU}=\cA_{\gamma_{k-1}, \UU}\circ \cdots\circ  \cA_{\gamma_{0}, \UU}.
\]
Hence, by \eqref{eq:bk_cA} we get
\begin{align*}
b_k'=\cA_{\gamma_{n-2}}\circ \cdots  \circ \cA_{\gamma_{0}}(\cA_{\gamma_0}^{-1}\circ \cdots\circ  \cA_{\gamma_{n-k-2}}^{-1}(0))&=\cA_{\gamma_{n-2}}\circ \cdots  \circ \cA_{\gamma_{n-k-1}}(0)\\
&=\cA_{\gamma_{n-2}^\iota}^{-1} \circ \cdots \circ  \cA_{\gamma_{n-k-1}^\iota}^{-1}(0).
\end{align*}
Equation \eqref{eq:bk_cA} now shows that $b_0', \dots, b_{n-1}'$ are the first $n$ elements of the path produced by a sequence of modified Verblunsky coefficients starting with $\gamma^\iota_{n-2},\ldots, \gamma^\iota_0$.  
\end{proof}

The following lemma describes the distribution of $\alpha^\iota$ if $\alpha$ has rotationally invariant distribution on $\UU$. 
\begin{lemma} \label{lem:iota}
Assume that $\alpha \in \UU$ is random with a rotationally invariant distribution. Then  the law of $\alpha^\iota$ is absolutely continuous with respect to the law of $\alpha$ with density $\Poi(\alpha, 1)$.
\end{lemma}
\begin{proof}
By first conditioning on $|\alpha|=r$, we may assume without loss of generality that  $\alpha=r \eta$ where $r\in (0,1)$ is deterministic,  and $\eta$ is uniform on the unit circle. (Note that $\alpha=0$ exactly if $\alpha^{\iota}=0$, and $\Poi(0,1)=1$.)
Then 
$$
\alpha^\iota=-\alpha\frac{1-\bar \alpha}{1-\alpha}=-r \eta\frac{1-r \eta^{-1}}{1-r \eta}=rg_r(\eta), \qquad 
g_r(u)=\frac{u-r}{ru-1}.$$
Since $\alpha\mapsto \alpha^\iota$ is an involution, we have $g_r^{-1}=g_r$. By the change of variables formula, the density of 
$\alpha^\iota$ at $ru$ with respect to uniform distribution on $r\partial  \UU$ is given by 
\[
|(g_r^{-1})'(u)|=|g_r'(u)|=\left|\frac{r^2-1}{(1-ru)^2}\right|=\Poi(ru,1).
\qedhere\]
\end{proof}

If the modified Verblunsky coefficients of a finite discrete probability measure on $\partial \UU$ satisfy certain invariance conditions then we can explicitly describe the affect of biasing by the weight of zero. This is the content of our next proposition.

\begin{proposition} \label{prop:path} Consider a random probability measure $\mu$ supported on  $n$ points on the unit circle. Let $\gamma_0, \dots, \gamma_{n-1}$ be its modified Verblunski coefficients, and $b_0, \dots, b_n$ its path parameter in $\UU$. 
Assume that for every $\eta\in \partial \UU$, the measures  $\mu_\eta$ and $\mu$ have the same law. Assume further that $\gamma_{n-1}\in \partial \UU$ is uniform and independent of the rest of the modified Verblunski coefficients. 

Then  $\mu$ biased by the weight of 1 exists; call it $\nu$. Denote the modified Verblunski coefficients of $\nu$ by $\hat \gamma_0, \dots, \hat \gamma_{n-1}$.  Then $\hat \gamma_{n-1}=1$ and the distribution of $(\hat \gamma_0,\ldots,\hat \gamma_{n-2})$ can be described in several ways as follows. 

\begin{enumerate}[(1)]
    \item The law of $(\gamma_0, \dots, \gamma_{n-2})$ given $b_n=1$.
    \item The law of $(\gamma_0, \dots, \gamma_{n-2})$  biased by $\Poi(b_{n-1},1)$.
    \item The law of $(\gamma_0, \dots, \gamma_{n-2})$  biased by $\prod_{i=0}^{n-2} \Poi(\gamma_i,1)$.
    \item If we further assume that the arguments of $\gamma_i$ are conditionally independent and uniform given the moduli, then the law of  $(\gamma_0^\iota, \dots, \gamma_{n-2}^\iota)$.
    \item Under (4), the law of the first $n-1$ of the deformed Verblunsky coefficients corresponding to the reversed path parameter 
$$
b'_k=
\cA_{\tilde b_{n-1},\UU}( \tilde b_{n-k-1}) \qquad k<n
$$
where $\tilde b$  is the path parameter for the coefficients $\gamma_{n-2},\ldots, \gamma_0$.
\end{enumerate}
Moreover, if $\mu(\cdot u)\ed \mu$ for all $u$, then $\nu$ is the Palm measure of $\mu$.
\end{proposition}
\begin{proof}
Consider the path parameter $b_0, \dots, b_n$
of $\mu$. By Corollary \ref{cor:cond_1} $\mu$ biased by the weight of 1 exists, and it is the same as $\mu_{b_n}$. Hence $\nu\ed\mu_{b_n}$.

Let $\alpha_0, \dots, \alpha_{n-1}$ be the Verblunsky coefficients of $\mu$, and let $\eta$ be uniform on $\partial \UU$, independent of $\alpha_0, \dots, \alpha_{n-1}$. 
Then  $\eta b_n\alpha_0,\ldots \eta b_n\alpha_{n-1}$ has the same distribution as $\alpha_0,\ldots ,\alpha_{n-1}$. Moreover, by Lemma \ref{lem:Aleks_charge} the  path parameter $b_{\eta b_n,0}, \dots, b_{\eta b_n,n}$ corresponding to the Verblunsky coefficients  $\eta b_n\alpha_0,\ldots \eta b_n\alpha_{n-1}$ satisfies $b_{\eta b_n,n}=1/\eta$.
So conditioning $\eta b_n\alpha_0$ on $b_{\eta b_n,n}=1$ just sets $\eta=1$. This means that 
$(\alpha_0, \dots, \alpha_{n-1})$ conditioned on $b_n=1$ has the same distribution as $(b_n \alpha_0, \dots,b_n \alpha_{n-1})$, the Verblunsky coefficients of $\mu_{b_n}$. Since the modified Verblunsky coefficients are functions of the regular Verblunsky coefficients (see \eqref{eq:gamma_alpha}), it follows that the distribution of $(\hat \gamma_0,\ldots,\hat \gamma_{n-2})$ is the same as the law described in (1).


Since $\gamma_{n-1}$ is independent of the rest, the conditional distribution of $b_n$ given $\gamma_0,\ldots \gamma_{n-2}$ is the same as the distribution of $\frac{b+e^{i \Theta}\frac{1-b}{1-\bar b}}{1+\bar b e^{i \Theta} \frac{
1-b}{1-\bar b}}$ with $\Theta$ uniform on $\partial \UU$ and $b=b_{n-1}(\gamma_0, \dots, \gamma_{n-1})$. A direct change of variables computation shows that this distribution has density 
$\Poi(b,\cdot)$.  Bayes' formula now shows that the distributions in (1) and (2) are the same.

Using \eqref{eq:Ay} and \eqref{eq:Ab_rec} we have
\begin{align*}
  \Poi(b_{n-1},1)=\det A_{b_{n-1}, \UU}=\det (A_{\gamma_{n-2}, \UU} \cdots A_{\gamma_{0},\UU})=\prod_{i=0}^{n-2} \det A_{\gamma_i, \UU} =\prod_{i=0}^{m-2}  \Poi(\gamma_i,1),
\end{align*}
which shows the equivalence of the distributions in (2) and (3). 

The equivalence of the distributions of (3) and (4) follows from Lemma \ref{lem:iota}. The equivalence of the distributions in  (4 and (5) is a consequence of Lemma \ref{lem:reverse}.
%
The last  claim follows from Proposition \ref{prop:Palm} and  Corollary \ref{cor:cond_1} .
\end{proof}

We finish the paper with the proof of Proposition \ref{prop:biased_V}.

\begin{proof}[Proof of Proposition \ref{prop:biased_V}]
The Verblunski coefficients $\gamma_0,\dots, \gamma_{n-1}$ of $\mu=\mu_{n,\beta}^{\textup{KN}}$ are independent, rotationally invariant, with density given by \eqref{eq:verb_circ}. Hence $\mu$ satisfies the conditions of Proposition \ref{prop:path}, including the additional condition of statement (4) and the condition $\mu(\cdot u)\ed \mu$ for all $u$. Hence by statement (3) of Proposition \ref{prop:path} the Verblunsky coefficients $\gamma_i'$ of $\nu$ are given by \eqref{eq:verb_nu} and $\gamma'_{n-1}=1$. It also follows that the Palm measure of $\mu$ is $\nu$. From Definition \ref{def:palm} and Remark \ref{rem:circle} the measure $\mu(X \cdot)$ is equal to the Palm measure of $\mu$. Hence $\mu(X\cdot)$ and   $\nu(\cdot)$ have the same law. Finally,  the weights of $\mu$ are independent of the support of $\mu$, and their expectations are equal to $1/n$, hence $\operatorname{supp} \mu(X\cdot)$ and  $\operatorname{supp} \tilde \mu(Y\cdot)$
have the same law as well. 
\end{proof}

\noindent {\bf Acknowledgments.}
The second author was partially supported by  the NSERC Discovery grant program.

$ $\\

\bigskip\noindent
Benedek Valk\'o
\\Department of Mathematics
\\University of Wisconsin - Madison
\\Madison, WI 53706, USA
\\{\tt valko@math.wisc.edu}
\\[20pt]
B\'alint Vir\'ag
\\Departments of Mathematics and Statistics
\\University of Toronto
\\Toronto ON~~M5S 2E4, Canada
\\{\tt balint@math.toronto.edu}

\end{document}